\documentclass[10pt,a4paper]{article}
\usepackage[tbtags]{amsmath}    
\usepackage{amsfonts,amssymb,amsthm,euscript,makeidx,pifont,boxedminipage,fullpage}
\usepackage{mathtools}
\usepackage[latin1]{inputenc}
\usepackage{ucs}
\usepackage[english]{babel}
\usepackage{amsmath}

\newcommand{\ds}{\displaystyle}

\newcommand{\bc}{\begin{center}}
\newcommand{\ec}{\end{center}}
\newcommand{\1}{\mathbf{1}}
\newcommand{\E}{\mathbb{E}}

\renewcommand{\P}{\mathbb{P}}

\newcommand{\R}{\mathbb{R}}

\newcommand{\N}{\mathbb{N}}

\newcommand{\cB}{\mathcal{B}}

\newcommand{\cD}{\mathcal{D}}
\newcommand{\cE}{\mathcal{E}}

\newcommand{\cL}{\mathcal{L}}

\newcommand{\cP}{\mathcal{P}}

\newcommand{\pare}[1]{\left ( #1 \right )}
\newcommand{\croc}[1]{\left [ #1 \right ]}

\newcommand{\espx}[1]{\mathbb{E}^{x_0}#1}
\newcommand\indi[1]{{\mathbb{I}}_{\mathnormal{#1}}}

\newtheorem{theorem}{Theorem}[section]
\newtheorem{lemma}[theorem]{Lemma}
\newtheorem{proposition}[theorem]{Proposition}
\newtheorem{corollary}[theorem]{Corollary}

\newtheorem{remark}[theorem]{Remark}

\newtheorem{assumption}[theorem]{Assumption}

\usepackage{color}

\reversemarginpar

\makeatletter

\@addtoreset{equation}{section}
\makeatother

\begin{document}

\bibliographystyle{plain}

\title{
A transformed stochastic Euler scheme for multidimensional transmission PDE
}

\date{\ }

\author{Pierre \'ETOR\'E \\
 LJK - Bâtiment IMAG\\
 UMR 5224\\
700, avenue centrale\\
38401 St Martin d'Hères, France\\
\\
\and Miguel MARTINEZ\footnote{Research supported by the LABEX Bézout}\\
Laboratoire d'Analyse et de Math\'ematiques Appliqu\'ees \\
UMR 8050\\
Universit\'e Paris-Est -- Marne-la-Vall\'ee \\
5 Boulevard Descartes \\
Cit\'e Descartes -- Champs-sur-Marne \\
77454 Marne-la-Vall\'ee cedex 2,
France
}

\maketitle

\begin{abstract}
\medskip
In this paper we consider multi-dimensional Partial Differential Equations (PDE) of parabolic type
in divergence form. The coefficient matrix of the divergence operator is assumed to be discontinuous along some smooth interface. At this interface, the solution of the PDE presents a compatibility transmission condition of its co-normal derivatives (multi-dimensional diffraction problem). We prove an existence and uniqueness result for the solution and study its properties.  In particular, we provide new estimates for the partial derivatives of the solution in the classical sense. We then construct a low complexity numerical Monte Carlo stochastic Euler scheme to approximate the solution of the PDE of interest.
 Using the afore mentioned estimates, we prove a convergence rate for our stochastic numerical method when the initial condition belongs to some iterated domain of the divergence form operator. 
Finally, we compare our results to classical deterministic numerical approximations and illustrate the accuracy of our method.

\end{abstract}

\section{Introduction}

Given a finite time horizon $T$,  a real valued function $x\mapsto u_0(x)$, and an elliptic symmetric matrix $x\mapsto a(x)\in \R^{d\times d}$, which is smooth except at the interface surface $\Gamma$ between subdomains $D_\pm$ of $\R^d$ ($\Gamma=\bar{D}_+\cap\bar{D}_-$), we consider the parabolic transmission (or diffraction) problem~:~find $u$ from $[0,T]\times \R^d$ to $\R$ satisfying
\begin{equation}
\label{edp-intro}
\left\{
\begin{array}{l}
\partial_tu(t,x)-\nabla.\pare{a(x)\nabla_x u(t,x)} = 0,~\;\forall (t,x)\in (0,T]\times (\R^d \setminus \Gamma)\\
\\
u(0,x)=u_0(x),\;\forall x\in \R^d\\
\\
\langle a_+\nabla_x u_{+}(t,.)- a_-\nabla_x u_{-}(t,.), \nu\rangle=0 \text{ and } u_+(t,.)=u_-(t,.)\;\text{ along }\Gamma.\\
\end{array}
\right.
\end{equation}

The objective of this paper is to provide an efficient stochastic numerical resolution method for the solution of \eqref{edp-intro}.

Parabolic equations involving $\nabla\cdot\pare{a\nabla}$ have been a major preoccupation for mathematicians in the fifties and the sixties. We may cite the pioneering works of  \cite{Nash-1957, Nash-1958},  \cite{De-Giorgi-1957}, and  \cite{Moser-1963, Moser-1964, Moser-1967} that prove the continuity of the solution of the Cauchy problem attached to $\nabla\cdot\pare{a\nabla}$ and also the celebrated paper by \cite{Aronson-1967}, which gives upper and lower Gaussian estimate bounds for the fundamental solution of the operator $\nabla\cdot\pare{a\nabla}$ (for a more modern perspective on evolution PDEs involving divergence form operators of type $\nabla\cdot\pare{a\nabla}$ see also \cite{Lions-Magenes-1972}). In these references assumptions on $a$ are very weak (it is assumed to be measurable, bounded and elliptic).

In the case where the matrix $a$ is assumed to be discontinuous along the regular boundaries of some nice disjoint connected open sets in $\R^d$, but smooth elsewhere, a refined analysis of the parabolic equation may be found in the monograph \cite{Lady-et-al-1967}. The authors interpret the parabolic equation  as a diffraction problem with transmission conditions along the discontinuity boundaries, of the type of \eqref{edp-intro}, and investigate the
classical smoothness of its solution.

\vspace{0.3cm}

When the underlying space is one-dimensional and the discontinuity is at zero ($\Gamma$ then reduces to the single point $0$), the link between 
\eqref{edp-intro} and some asymmetric diffusion process $X$ is well known. More precisely one has that
$
u(t,x)=\E^x[u_0(X_t)]$ 
where $X$ is solution to the Stochastic Differential Equation (SDE) with local time
\begin{equation}
\label{eq:SDE-local-time}
dX_t = \sigma(X_t)dW_t + a'(X_t)dt + \frac{a(0+) - a(0-)}{a(0+)+a(0-)}dL_t^0(X)
\end{equation}
where $\sigma^2 = 2a$ and $a'$ denotes a function that coincides with the first order derivative of $a$  outside zero, and can be set at any arbitrary value at zero.
In \eqref{eq:SDE-local-time} we have denoted $W$  a standard one-dimensional Brownian motion (B.m.), and $L_t^0(X)$ the symmetric local time of $X$ at time $t$. 
Under mild conditions \eqref{eq:SDE-local-time} has a unique strong solution $X$, see \cite{LeGall-1984}.
Put in other words the operator $\nabla\cdot\pare{a\nabla}$ appears as the infinitesimal generator of the diffusion $X$ solution of \eqref{eq:SDE-local-time}. Note that the local time term in 
\eqref{eq:SDE-local-time} is a singular term that reflects the discontinuity of $a$ along $\Gamma=\{0\}$.

 For a study of the one-dimensional case one may refer to the overview \cite{Lejay-2006}, \cite{Etore-Martinez-2017}, and the series of works \cite{Martinez-2004, Martinez-Talay-2006, Martinez-Talay-2012,Lejay-Martinez-2006, Etore-2006, Etore-Lejay-2007, Etore-Martinez-2013,Etore-Martinez-2014} \cite{lejay-maire2013,mazzonetto2,frikha2018,lenotre-2019} where stochastic numerical schemes are presented. 
 
 Then if one constructs a scheme $\overline{X}$ approaching $X$ (in law for example), we will have that
 $\E^x[u_0(\overline{X}_t)]$ approaches $u(t,x)$. This provides some stochastic numerical resolution method for the solution of \eqref{edp-intro}. But none of the above cited works on the one-dimensional case can be directly adapted to the multidimensional case.
 


\vspace{0.3cm}

As a matter of fact, till now and up to our knowledge much fewer stochastic schemes have been proposed and studied to tackle the multidimensional case.

A natural idea would be to regularize the coefficient $a$ around the interface $\Gamma$, and then to perform a discretized stochastic scheme on the classical problem obtained by regularization (for smooth $a^\varepsilon$ the process $X$ in link with $\nabla\cdot\pare{a^\varepsilon\nabla}$ is some Itô process with classical drift that can be approached by a standard Euler scheme). But then there is a balance to find between the regularization step and the discretization step. Such methods are less precise and less investigated (see \cite{stroock1997} for some elements in this direction; see also some of our numerical results in Section  \ref{sec:num}).

For some results with no regularization procedure
 see   \cite{Lenotre-2015}, and  \cite{Bossy-al-2010} in the case of a diagonal coefficient matrix $a$ constant outside the discontinuity boundary $\Gamma$; see also \cite{Limic-2011}, which attempts to interpret stochastically the deterministic Galerkin method using jump Markov Chains. 

In this paper we will propose a stochastic numerical scheme that allows to treat the multidimensional  case, when the matrix-valued diffusion coefficient $a$ is not necessarily diagonal, nor piecewise constant. We aim at treating the discontinuity of $a$ directly and use no regularization. The scheme we propose is of Euler type; it can be seen as en extension to the multidimensional case of the scheme studied in \cite{Martinez-Talay-2012} (see some comments in
Remark \ref{rem:euler-1D}).

\vspace{0.3cm}
One of the difficulties of the multidimensional case is that the stochastic process $X$ naturally in link with the operator $\nabla\cdot\pare{a\nabla}$ is more difficult to describe than in dimension one. One knows that the operator generates such a process $X$, which is Markov (see for instance \cite{Stroock-1988}; on the Dirichlet form approach see \cite{Fukushima-et-al-2011}, in particular Exercise 3.1.1 p111). 
We still have the link $u(t,x)=\E^x[u_0(X_t)]$.
But the Itô dynamic of $X$ is difficult to establish and to exploit. In the companion paper \cite{etore-martinez6} we have been able to prove (in the case $\R^d=\bar{D}_+\cup D_-$ and $a$ has some smoothness in $D_\pm$) that
\begin{align}
\label{EDS:multi}
X_t^k &= x_k + \int_0^t \sum_{j=1}^d \sigma_{kj}(X_s)dW_s^{j} + \int_{0}^t \sum_{j=1}^d \partial_j a_{kj}(X_s)\indi{X_s\in D}ds\nonumber\\
&\hspace{0.5 cm} - \frac{1}{2}\int_0^t \gamma_{+,k}(X_s) dK_s + \frac{1}{2}\int_0^t \gamma_{-,k}(X_s)dK_s,\hspace{0,3 cm}t\geq 0.
\end{align}
In this expression $W$ is a standard B.m., we have $\sigma\sigma^*=2a$, the terms $\gamma_{\pm,k}$ are some co-normal vectors to the surface $\Gamma=\bar{D}_+\cap\bar{D}_-$, and
$K$ is the PCAF associated through the Revuz correspondence to the surface measure on $\Gamma$. Equation \eqref{EDS:multi} is in some sense the multi-dimensional analog to
\eqref{eq:SDE-local-time}, the singular term being now $- \frac{1}{2}\int_0^t \gamma_{+,k}(X_s) dK_s + \frac{1}{2}\int_0^t \gamma_{-,k}(X_s)dK_s$. 

However to infer from \eqref{EDS:multi} an approximation scheme $\overline{X}$ for $X$ is not easy. In the afore mentioned works about the one-dimensional case, things are most often achieved 
with the help of Itô-Tanaka type formulas, that allow to manipulate SDEs with local time. In the multi-dimensional case we have not access to such a formula, and in addition we know less about the singular term
$- \frac{1}{2}\int_0^t \gamma_{+,k}(X_s) dK_s + \frac{1}{2}\int_0^t \gamma_{-,k}(X_s)dK_s$ than we know about $L^0_t(X)$ in the one-dimensional case.

\vspace{0.3cm}
Thus we are led, in the present paper, to contruct a stochastic scheme $\overline{X}$ such that $\E^x[u_0(\overline{X}_t)]$ approaches $u(t,x)$, but {\it without seeking to approach
$X$ by~$\overline{X}$.}  The idea will be to perform a standard Euler scheme as long as $\overline{X}$ does not cross the boundary $\Gamma$. But when the scheme $\overline{X}$ crosses the boundary we will correct its position in a way that reflects the transmission condition in \eqref{edp-intro}. The contribution of the paper are the following.

\vspace{0.1cm}
I) We will first study the PDE \eqref{edp-intro}. We will show that, when the initial condition $u_0$ belongs to some iteration of the domain of $\nabla\cdot\pare{a\nabla}$, this PDE has a classical solution. Then we prove the existence of global bounds for the partial derivatives of this solution (up to order four in the space variable) outside the discontinuity boundary $\Gamma$, for all strictly positive times (and not just for times $t$ satisfying $t\geq \varepsilon$ for some $\varepsilon>0$), and all the way up to the boundary (not only interior estimates). In our opinion these estimates are new (compared to \cite{Lady-et-al-1967}) and have an interest per se. These estimates will be needed to perform the convergence analysis of our scheme. The method we follow, in this PDE oriented part of the paper, is combining the Hille-Yosida theorem with
results on elliptic transmission PDEs to be found in \cite{McLean-2000}. 

\vspace{0.1cm}
II) We propose our scheme (we insist again that our matrix valued coefficient $a$ does not need to be diagonal) and study its convergence rate. More precisely we prove that
$$
\Big| u(T,x_0)-\E^{x_0}u_0(\overline{X}^n_T)  \Big|\leq K \sqrt{h_n},
$$
where $h_n$ is the time step of our Euler scheme (see the precise assumptions and statement in Theorem \ref{thm-conv-schema}). This has to be compared with the results in the classical smooth case and the ones in \cite{Martinez-Talay-2012}.

\vspace{0.3cm}
The paper is organised as follows. In Section \ref{sec:notations}  we present the notations of the paper and our main assumptions. In Section \ref{sec:EDP} we define precisely and study the parabolic transmission problem \eqref{edp-intro}, proving in particular an existence and uniqueness result for a classical solution, for which we get estimates for the space and time derivatives.
In Section \ref{sec:euler} we present our scheme, and in Section \ref{sec:conv-euler} we analyse its convergence. Section  \ref{sec:num} is devoted to numerical experiments.

\section{General notations and assumptions}
\label{sec:notations}

For two points $x,y\in\R^d$ we denote by $\langle x,y\rangle$ their scalar product $\langle x,y\rangle=x^* y= \sum_{i=1}^dx_iy_i$.

For a point $x\in\R^d$ we denote by $|x|$ its Euclidean norm i.e. $|x|^2=\sum_{i=1}^dx_i^2=\langle x,x\rangle$.

We denote by $(e_1,\ldots,e_d)$ the usual orthonormal basis of $\R^d$.

For two metric spaces $E,F$ we will denote by $C(E;F)$ the set of continuous  functions from $E$ to $F$ and, for
$1\leq p\leq \infty$, by  $C^p(E;F)$ the set of functions in $C(E;F)$ that are $p$ times differentiable with continuous derivatives.

We will denote by $C^p_c(E;F)$ the set of functions in $C^p(E;F)$ that have a compact support.

We will denote by $C^p_b(E;F)$ the set of functions in $C^p(E;F)$ that are continuous with bounded $p$ first derivatives 
($C_b(E;F)$ denotes the set of functions in $C(E;F)$ that are bounded).

If $F=\R$, we will sometimes simply write for instance $C(E)$ for $C(E;\R)$, for the sake of conciseness.

For any multi-index $\alpha = (i_1,\dots,i_d)\in \N^d$ and $x=(x_1,\dots,x_d)\in \R^d$, we note $x^\alpha$ the product $x_1^{i_1}\dots x_d^{i_d}$ and $|\alpha| = i_1 + \dots + i_d$.  So that for $u\in C^{|\alpha|}(\R^d)$ we will denote
$\frac{\partial^{|\alpha|}u}{\partial x^\alpha}$, or in short $\partial^\alpha u$, the partial derivative 
$\partial^{i_1}_{x_1^{i_1}}\ldots\partial^{i_d}_{x_d^{i_d}}u$.

\vspace{0.2cm}

Let $U\subset \R^d$ an open subset. We will denote by $L^2(U)$ the set of square integrable functions from $U$ to $\R$ equipped with the usual norm and scalar product 
$||\cdot||_{L^2(U)}$ and $\langle \cdot,\cdot\rangle_{L^2(U)}$.

We denote $H^1(U)$ the usual Sobolev space $W^{1,2}(U)$, equipped with the usual norm $||\cdot||_{H^1(U)}$.
We will denote by $D_iv$ the derivative in the distribution sense with respect to $x_i$ of $v\in L^2(U)$. 

 We recall that the space $H^1_0(U)\subset H^1(U)$ can be defined as
$H^1_0(U)=\overline{C^\infty_c(U;\R)}=\overline{C^1_c(U;\R)}$. 

We denote $H^{-1}(U)$ the usual dual topological space of $H^1_0(U)$.

For $m\geq 2$, we denote $H^m(U)$  the usual Sobolev space $W^{m,2}(U)\subset L^2(U)$ of functions having $m$ successive weak derivatives in $L^2(U)$.

\vspace{0.2cm}

The notion of a $C^k$ domain $U\subset\R^d$  with bounded boundary $\Gamma=\partial U$  is defined with the help of a system of local  change of coordinates  of class $C^k$ (see \cite{McLean-2000} Chap.3 pp89-90).

From now on we consider in the whole paper that $\R^d=\bar{D}_+\cup D_-$ with $D_+$ and $D_-$ two open connected subdomains separated by a transmission boundary~$\Gamma$ that is to say$$\Gamma=\bar{D}_+\cap \bar{D}_-$$ 
(in addition we will denote 
$D=D_+\cup D_-=\R^d\setminus\Gamma\subset\R^d$).

By an assumption of type "$\Gamma$ is bounded and  $C^k$" we will mean that both $D_+$ and $D_-$ are  $C^k$ domains, and that $\Gamma$ is bounded. Note  that in that case we shall consider $D_+$  (resp. $D_-$) as the interior (resp. exterior) domain. Note that $D_-$ is then unbounded (although its boundary is bounded).

\vspace{0.2cm}
Assume $\Gamma$ is bounded and $C^2$. We will denote $\gamma: H^1(D_\pm)\to H^{1/2}(\Gamma)$ the usual trace operator on $\Gamma$ and $H^{-1/2}(\Gamma)$ the dual space of
$H^{1/2}(\Gamma)$ (see p98-102 in \cite{McLean-2000}).

\vspace{0.4cm}
In the sequel we will frequently note $f_\pm$ the restrictions of a function $f$ to $D_\pm$. Besides, by an assumption of type 
"the function $f$ satisfies $f_\pm\in C^p(\bar{D}_\pm)$" (or "$f\in C^p(\bar{D}_+)\cap C^p(\bar{D}_-)$") we will mean that the restriction of $f$ to $D_+$
(resp. $D_-$) coincides on $D_+$ (resp. $D_-$) with a function $\tilde{f}_+$ of class $C^p(\R^d)$ (resp. $\tilde{f}_-$).
So that for any $x\in\Gamma$ we can give a sense for example to $f_+(x)$: it is
$\lim_{z\to x\,,\,z\in D_+} f(z)=\tilde{f}_+(x)$. 

In the same time spirit we may note for $f\in C(\bar{D}_+)\cap C(\bar{D}_-)$ and a point $y\in\Gamma$
\begin{equation*}
f(y\pm)=\lim_{z\to y\,,\,z\in D_\pm}f(z)=f_\pm(y).
\end{equation*}


For $u\in C^1(\bar{D}_+;\R)\cap C^1(\bar{D}_-;\R)$ we denote $\nabla_xu=(\frac{\partial u}{\partial x_1},\ldots,\frac{\partial u}{\partial x_d})^\ast$ and,
for a point $y\in\Gamma$
\begin{equation}
\label{def-nabla-pm}
\nabla_xu_\pm(y)=\lim_{z\to y\,,\,z\in D_\pm}\nabla_xu(z).
\end{equation}
For a vector field $G\in  C^1(D;\R^d)$ we denote by $\nabla\cdot G\,(x)$ its divergence at point $x\in D$, i.e. 
$\nabla\cdot G(x)=\sum_{i=1}^d\frac{\partial G_i}{\partial x_i}(x)$.

For $u\in C^2(D;\R)$ and $x\in D$ we denote $\mathbf{H}[u](x)$ the Hessian matrix of $u$ at point $x$.

\vspace{0.5cm}
Let $a(x) = (a_{ij}(x))_{i,j\in \{1,\dots, d\}}$ be a symmetric matrix valued and time homogeneous diffusion coefficient.

\vspace{0.1cm}
If $a_{ij}\in C^1(D;\R)$ for all $1\leq i,j\leq d$ and $u\in  C^2(D;\R)$ we denote
\begin{equation}
\label{eq-def-calL}
{\cal L}u(x) = \nabla\cdot\pare{a(x)\nabla_x u(x)}, \quad\forall x\in D.
\end{equation}

 In the whole paper the coefficients of the function matrix $a$ are always assumed to be measurable and bounded by a constant $\Lambda$.

\vspace{0.2cm}

We will also often make the following ellipticity assumption
\begin{assumption} \;{\rm ({\bf E})~:}
There exists $\lambda\in (0,\infty)$ such that 
\begin{equation}
\label{eq:ell}
\forall x\in \R^d,\;\;\forall \xi\in\R^d,\quad \lambda|\xi|^2\leq \xi^\ast a(x)\xi.
\end{equation}
\end{assumption}

%

\vspace{0.1cm}
Note that under $(\mathbf{E})$ we can assert that for any $x\in D$ we have
\begin{equation}
\label{eq:decomp-ortho}
a_\pm(x)=P_\pm^*(x)E_\pm(x)P_\pm(x)
\end{equation}
with $P_\pm(x)$ some orthogonal matrices and $E_\pm(x)$ some diagonal matrices with strictly positive eigenvalues.

\vspace{0.4cm}
Assume $\Gamma$ is $C^2$. For a point $x\in\Gamma$ we denote by $\nu(x)\in\R^d$ the unit normal to $\Gamma$ at point $x$, pointing to $D_+$. 
Assume the $a_{ij}$'s satisfy $(a_\pm)_{ij}\in C(\bar{D}_\pm)$.
We define then the co-normal vector fields ${\gamma_+}(x) := a_+(x)\nu(x)$ and $\gamma_-(x) := -a_-(x)\nu(x)$,
for $x\in\Gamma$.


\vspace{0.2cm}

Note that under $(\mathbf{E})$ it is clear that we have
\begin{equation}
\label{eq:controle-conom}
\forall x\in  \Gamma,\hspace{0.3 cm} \langle {\gamma_+}(x), \nu(x)\rangle \geq \lambda >0
\hspace{0.3 cm} \text{and}\hspace{0.3cm}\langle \gamma_-(x), \nu(x)\rangle \leq  - \lambda <0.
\end{equation}

Note that the notation $\gamma$ for the trace operator follows the usual one (\cite{McLean-2000} for instance) and the notation~$\gamma_\pm$ for the co-normal vectors follows the one of the paper \cite{Bossy-et-al-2004}. But it will be dealt with the trace operator only in Section \ref{sec:EDP}, and with co-normal vectors only in Sections \ref{sec:euler} and \ref{sec:conv-euler}. So that these notations will cause no confusion.

\vspace{0.3cm}

To finish with we define the unbounded operator $A:\cD(A)\subset L^2(\R^d)\to L^2(\R^d)$ by
$$
\cD(A)=\big\{  u\in H^1(\R^d)\text{ with }\sum_{i,j=1}^dD_i(a_{ij}D_ju)\in L^2(\R^d)  \big\}
$$
and
$$
\forall u\in\cD(A), \quad Au=\sum_{i,j=1}^dD_i(a_{ij}D_ju).$$
Then we introduce the iterated domains defined recursively by
 $$
 \cD(A^k)=\{ v\in\cD(A^{k-1}):\;\;Av\in \cD(A^{k-1})\},\quad k\geq 2.$$

\section{The parabolic transmission problem}
\label{sec:EDP}

Let $0<T<\infty$ a finite time horizon. Let us consider the transmission parabolic problem

$$
(\mathcal{P}_\mathrm{T})
\left\{
\begin{array}{rcll}
\partial_tu(t,x)-\mathcal{L}u(t,x)&=&0&\forall (t,x)\in (0,T]\times D \\
\\
\langle a_+\nabla_x u_{+}(t,y)- a_-\nabla_x u_{-}(t,y), \nu(y)\rangle &=& 0&\forall (t,y)\in(0,T]\times\Gamma\quad (\star) \\
\\
u(t,y+)&=&u(t,y-)&\forall (t,y)\in[0,T]\times\Gamma\\
\\
u(0,x)&=&u_0(x)&\forall x\in \R^d.\\
\end{array}
\right.
$$

\vspace{0.5cm}
\noindent
We will say that $(t,x)\mapsto u(t,x)$ is {\it classical solution} to $(\cP_{\mathrm{T}})$ if it satisfies
  \begin{equation}
  \label{eq:esp-sol-classique}
  u\!\in\!C\big([0,T]; C^2(\bar{D}_+)\cap C^2(\bar{D}_-)\big ) \cap C^1\big([0,T]; C(\bar{D}_+)\cap C(\bar{D}_-)\big ) \cap C\big([0,T]; C(\R^d)\big )
  \end{equation}
   and satisfies the following requisites. First, $u$ satisfies the first line of $(\mathcal{P}_{\mathrm{T}})$, where the derivatives are understood in the classical sense. Second,
 for all $0<t\leq T$  the limits
$\lim_{z\to y\,,\,z\in D_\pm}\nabla_xu(t,z)$  satisfy the transmission condition $(\star)$ for all $y\in\Gamma$.
Note that these limits exist thanks to \eqref{eq:esp-sol-classique}.
Third, $u$ is continuous accross $\Gamma$ (third line).  Fourth, it satisfies the initial condition at the fourth line of $(\cP_{\mathrm{T}})$.

\noindent
The aim of this section is to prove the following result.

\begin{theorem}
\label{thm-existe-sol-classique}
Let $a=(a_{ij})_{1\leq i,j\leq d}$ satisfy $(\mathbf{E})$. 
\begin{itemize}
\item Denote 
\begin{equation}
k_0=\left \{\begin{array}{l}
\lfloor{\frac{d}{4}}\rfloor + 2\;\;\text{ if }d\text{ is even;}\\
\lfloor{\frac{3}{2} + \frac{\lfloor{d/2}\rfloor}{2}}\rfloor + 2\;\;\text{ if }d\text{ is odd.}
\end{array}\right.
\end{equation}

Assume that the coefficients $a_{ij}$ satisfy $(a_\pm)_{ij}\in C_b^{2k_0-3}(\bar{D}_\pm)$ and $\Gamma$ is bounded and of class $C^{2k_0-2}$. 
Then for~$u_0\in \cD(A^{k_0})$ the parabolic transmission problem $(\mathcal{P}_{\mathrm{T}})$ admits a classical solution. 
\item Furthermore, if $u_0\in\cD(A^k)$ for $k\geq k_0$, the coefficients $a_{ij}$ satisfy $(a_\pm)_{ij}\in C_b^{2k-1}(\bar{D}_\pm)$ and 
$\Gamma$ is bounded of class $C^{2k}$,
this classical solution $u$ is such that
$$
u\in C^{k-j}\left([0,T]\;;\;C^{n(j)}(\bar{D}_+)\cap C^{n(j)}(\bar{D}_-)\right),\quad \lceil d/4\rceil\leq j\leq k$$
with $n(j) = \lfloor 2j - \frac{d}{2}\rfloor$.
\end{itemize}
\end{theorem}

To prove Theorem \ref{thm-existe-sol-classique} requires to study in a first time the associated elliptic resolvent equation, in a weak sense.
More precisely, for a source term $f\in L^2(\R^d)$ we will seek for a solution $u$ in $\cD(A)$ of
\begin{equation}
\label{eq:resol}
u-Au=f
\end{equation}
(see Proposition \ref{prop:sol-faible-ell} below). 

Then the idea is to apply in $L^2(\R^d)$ a version of the Hille-Yosida theorem that states that for $u_0\in \cD(A^k)$, $k\geq 2$, there is a solution $u$  to
$\frac{\mathrm{d}u}{\mathrm{d}t}=Au,\;\; u(0)=u_0$, living in
$C^{k-j}\big([0,T];\, \cD(A^j)  \big),\; 0\leq j\leq k$ (see Proposition \ref{prop:hille} below).  

As we will have studied the weak smoothness of functions living in the $\cD(A^k)$'s (Proposition \ref{prop-sautzero} and Corollary \ref{cor:regDA}), we will be able to conclude, using Sobolev embedding arguments.

\begin{remark} 1)
 In the  classical situation with smooth coefficients studied for instance in \cite{Friedman-1964} Chap. 1 (or \cite{Lieberman-1996}, Theorem 5.14), a unique classical solution to the parabolic PDE exists as soon as the $a_{ij}$'s are bounded and H\"older continuous and satisfy $(\mathbf{E})$, and $u_0$ is continuous and satisfies some growth condition. 
 
 Here we ask additional smoothness on the coefficients
$(a_\pm)_{ij}$'s inside the domains $D_\pm$. Indeed, because of the discontinuity of $a$ across $\Gamma$ we are led to use a different technique of proof: unlike the parametrix method in the classical case, this additional smoothness is required for the use of the Hille-Yosida theorem and the Sobolev embeddings. 

Note that with this methodology of proof these additional assumptions would still be needed  if our coefficients and the solution were smooth at the interface. Note that with this approach the assumptions on the initial condition $u_0$ are understood in a weak sense (and are different).

2) Our result is also different from the one in \cite{Lady-et-al-1967} (Theorem 13.1; see also \cite{lady1}). In this reference the authors study the classical smoothness of the parabolic transmission problem by studying first the smoothness of $\partial_tu$ (to that aim they differentiate with respect to time the initial equation). Then they study the smoothness with respect to the space variable by using results for the elliptic transmission problem, involving difference quotient techniques. But by doing so they get estimates on subdomains of the form
$[\varepsilon,T]\times \overline{D}_\pm$ with $0<\varepsilon$.
Here, we manage to study the global regularity of the classical solution of ${\cal P}_T$ in the whole domains $[0,T]\times \overline{D}_\pm$.
\end{remark}

\subsection{Study of the associated elliptic problem and of the domains $\cD(A^k)$} 

In this subsection we establish the existence of a solution to \eqref{eq:resol} belonging to~$\cD(A)$ and study its smoothness properties, together with the ones of functions belonging to the iterated domains $\cD(A^k)$, for~$k\geq 1$.

\vspace{0.2cm}

We recall that the coefficients $a_{ij}$ are assumed to be bounded by $\Lambda$ so that we may define the following continuous bilinear and symmetric form, which will be used extensively in the sequel
\begin{equation}
\label{eq:forme-bilin}
\cE(u,v)= \sum_{i,j=1}^d\langle a_{ij}D_ju,D_iv\rangle_{L^2(\R^d)},\quad\forall u,v\in H^1(\R^d).
\end{equation}

Let $u\in \cD(A)$. Using the definition of $Au$ as a distribution acting on $C^\infty_c(\R^d;\R)$, and the density of $C^\infty_c(\R^d;\R)$
in $H^1(\R^d)=H^1_0(\R^d)$, one can establish the following relation, linking $A$ and the form \eqref{eq:forme-bilin}:
\begin{equation}
\label{eq:lienAE}
\cE(u,v)=\langle -Au,v\rangle_{L^2(\R^d)},\quad\forall v\in H^1(\R^d).
\end{equation}

\subsubsection{Some results on weak solutions of elliptic transmission PDEs} 

Here we gather some preliminary results on weak solutions of elliptic transmission PDEs that rely mainly on \cite{McLean-2000} Chap. 4, pp. 141-145.

We recall that for $u\in L^2(\R^d)$, we denote $u_+$ (resp. $u_-$) the restriction of $u$ to $D_+$ (resp.~$D_-$). It may happen that we use this notation for restricted distributions also.

We introduce the following notation for the jump across $\Gamma$ of $u\in L^2(\R^d)$, with $u_+\in H^1(D_+)$ and $u_-\in H^1(D_-)$:
$$
[u]_\Gamma=\gamma(u_+)-\gamma(u-).$$
If $[u]_\Gamma=0$ we shall simply write $\gamma(u)=\gamma(u_+)=\gamma(u_-)$.
We have the two following lemmas (the proof of the first one is straightforward).

\begin{lemma}
\label{lem:restri}
Let $v\in L^2(\R^d)$. Then, for any $1\leq i\leq d$, the distribution $(D_iv)_\pm$ is equal to $D_i(v_\pm)$. As a consequence, if $v\in H^1(\R^d)$,
then $v_\pm\in H^1(D_\pm)$.
\end{lemma}

\begin{lemma}[\cite{McLean-2000}, Exercise 4.5]
\label{lem:exo-mclean}
Suppose $u\in L^2(\R^d)$ with $u_\pm\in H^1(D_\pm)$. Then $u\in H^1(\R^d)$ if and only if $[u]_\Gamma=0$ a.e. on $\Gamma$.
\end{lemma}

We shall consider restricted operators and bilinear forms in the following sense. We define 
$A_+:H^1(D_+)\to H^{-1}(D_+)$ by
$$
\forall v\in H^1(D_+),\quad A_+v=\sum_{i,j=1}^dD_i\big((a_+)_{ij}D_jv\big).$$
We define $A_-:H^1(D_-)\to H^{-1}(D_-)$ in the same manner (note that we do not specify here any domain~$\cD(A_\pm)$). Further, we define
\begin{equation*}
\cE_\pm(u,v)= \sum_{i,j=1}^d\int_{D_\pm} (a_\pm)_{ij}D_ju\,D_iv,\quad\forall u,v\in H^1(D_\pm).
\end{equation*}
In the same fashion as for Equation \eqref{eq:lienAE}, we have, for $u_\pm\in H^1(D_\pm)$ with $A_\pm u_\pm \in L^2(D_\pm)$,
\begin{equation}
\label{eq:lienAEp}
\cE_\pm(u_\pm,v)=\int_{D_\pm}(-A_\pm u_\pm)v,\quad\forall v\in H^1_0(D_\pm).
\end{equation}

Imagine now that in \eqref{eq:lienAEp} we wish to take the test function in $H^1(D_\pm)$ instead of $H^1_0(D_\pm)$. There will still be a link between $A_\pm$ and $\cE_\pm$, but through Green type identities, involving 
co-normal derivatives and boundary integrals. We have the following result.

\begin{proposition}[First Green identity, extended version; see \cite{McLean-2000} Theorem 4.4, point i)]
\label{prop:green2}
Assume $\Gamma$ is bounded and $C^2$.
Let $u\in L^2(\R^d)$ with $u_+\in H^1(D_+)$ and $u_-\in H^1(D_-)$. Assume $A_+u_+\in L^2(D+)$, $A_-u_-\in L^2(D-)$. Then there exist uniquely defined
elements~$\cB^\pm_\nu u\in H^{-\frac 1 2}(\Gamma)$ such that
\begin{equation}
\label{eq:greenp}
\cE_+(u_+,v)=\int_{D_+}(-A_+u_+)v-\Big( \cB^+_\nu u,\gamma(v) \Big)_\Gamma \;\;,\quad\forall v\in H^1(D_+)
\end{equation}
and
\begin{equation}
\label{eq:greenm}
\cE_-(u_-,v)=\int_{D_-}(-A_-u_-)v+ \Big( \cB^-_\nu u,\gamma(v) \Big)_\Gamma \;\;,\quad\forall v\in H^1(D_-).
\end{equation}
\end{proposition}

The elements $\cB^\pm_\nu u$ in Proposition \ref{prop:green2} are the one-sided co-normal derivatives of $u$ on $\Gamma$. 

To fix ideas, note that
 under the stronger assumptions that the $(a_\pm)_{ij}$'s are in $C^1_b(\bar{D}_\pm;\R)$,  and $u_\pm\in H^2(D_\pm)$, we have
\begin{equation*}
\label{eq:derconorm}
\cB^{\pm}_\nu u=\nu^*\gamma(a_\pm\nabla u_\pm)=\sum_{i=1}^d\sum_{j=1}^d\nu_i\gamma\big((a_\pm)_{ij}D_ju_\pm\big)\quad\text{ on  }\;\;\Gamma
\end{equation*}
(note that as the $(a_\pm)_{ij}D_ju_\pm$'s are in $H^1(D_\pm)$ the trace terms are correctly defined in the above expression). Thus one understands that the change of sign in front
of the $(\cdot,\cdot)_\Gamma$ term
between \eqref{eq:greenp} and \eqref{eq:greenm} is due to the fact that $-\nu$ is the outward normal to $D_+$ and $\nu$ is the outward normal to $D_-$.

For details on the definition of 
$\cB^\pm_\nu u$ under the weaker assumptions of Proposition \ref{prop:green2}, see \cite{McLean-2000} pp 116-117.

\vspace{0.2cm}

Finally we introduce a notation for the jumps across $\Gamma$ of the co-normal derivative of a function $u$  satisfying the assumptions of Proposition \ref{prop:green2}:
$$
\big[ \cB_\nu u \big]_\Gamma=\cB^+_\nu u-\cB^-_\nu u \in  H^{-1/2}(\Gamma).$$
We have the following result.

\begin{lemma}[Two-sided Green identity; inspired by \cite{McLean-2000} Lemma 4.19, Equation (4.33)]
\label{lem:green3}
Assume $\Gamma$ is bounded and $C^2$.
Let $u\in H^1(\R^d)$. Let $f_+\in L^2(D_+)$ and $f_-\in L^2(D_-)$ and assume
\begin{equation}
\label{eq:respm}
u_\pm-A_\pm u_\pm=f_\pm\quad\text{on }D_\pm.
\end{equation}
Set $f=f_++f_-$, then
\begin{equation}
\label{eq:green2s}
\langle u,v\rangle_{L^2(\R^d)}+\cE(u,v)=\langle f,v\rangle_{L^2(\R^d)}-\Big(\big[ \cB_\nu u \big]_\Gamma\,,\,\gamma(v)\Big)_\Gamma\;\;,
\quad\forall v\in H^1(\R^d).
\end{equation}
\end{lemma}

\begin{remark}
\label{rem:restri}
Note that in the above proposition $u_\pm\in H^1(D_\pm)$, thanks to Lemma \ref{lem:restri}. Equation \eqref{eq:respm} means that 
 $\langle u_\pm-A_\pm u_\pm,\,\varphi\rangle_{H^{-1}(D_\pm),H^1_0(D_\pm)}=\langle f,\varphi\rangle_{L^2(D_\pm)}$, for all
 $\varphi\in C^\infty_c(D_\pm;\R)$. Therefore $A_\pm u_\pm\in L^2(D_\pm)$ and by Proposition \ref{prop:green2} the element $\big[ \cB_\nu u \big]_\Gamma$ is well defined.
Then same remark holds for the forthcoming Proposition \ref{prop:reg-sol-faible}.
\end{remark}

Our notations being different from the ones in \cite{McLean-2000}, we provide the short proof of Lemma \ref{lem:green3} for the sake of clarity.

\begin{proof}
 Taking into account Remark \ref{rem:restri} we can use Proposition \ref{prop:green2}, and summing
 \eqref{eq:greenp} and  \eqref{eq:greenm} one gets for any $v\in H^1(\R^d)$ (note that $\gamma(v_+)=\gamma(v_-)=\gamma(v)$)
 $$
\langle u,v\rangle_{L^2(\R^d)}+\cE_+(u_+,v_+)+\cE_-(u_-,v_-)=\langle f,v\rangle_{L^2(\R^d)}- \Big( \big[ \cB_\nu u \big]_\Gamma\,,\,\gamma(v) 
\Big)_\Gamma
.$$
To complete the proof it suffices to notice that, thanks to Lemma \ref{lem:restri}, we have
$$
\begin{array}{lll}
\cE_+(u_+,v_+)+\cE_-(u_-,v_-)&=&\ds\sum_{i,j=1}^d\Big\{ \int_{D_+}(a_+)_{ij}(D_ju)_+(D_iv)_+ \\
&&\hspace{2cm}
+ \int_{D_-}(a_-)_{ij}(D_ju)_-(D_iv)_- \Big\}\\
&=&\cE(u,v).
\end{array}$$
\end{proof}


We recall  now results on the smoothness of weak solutions of elliptic transmission PDEs.

\begin{proposition} [\cite{McLean-2000}, Theorem 4.20]
\label{prop:reg-sol-faible}

Let $G_1$ and $G_2$ be bounded open connected subsets of $\R^d$, such that $\overline{G_1}\subset G_2$ and $G_1$ intersects 
$\Gamma$, and put
$$
D^j_\pm=G_j\cap D_\pm\quad\text{and}\quad \Gamma_j=\Gamma\cap G_j\quad\text{for} \;\; j=1,2.$$
Assume that the set $G_2$ is constructed in such a way that there is a $C^{r+2}$-diffeomorphism between $\Gamma_2$ and a bounded portion of the hyperplan $x_d=0$.

Assume  $(\mathbf{E})$.

Let $r\in \N$. Assume that the coefficients $(a_{\pm})_{ij}$ belong to $C^{r+1}(\overline{D^2_\pm} ; \R)$. 

Let $f_\pm\in L^2(D_\pm)$ with $f_\pm\in H^r(D^2_\pm)$. Let $u\in L^2(\R^d)$ with $u \in H^1(G_2)$ satisfying 
$$u_\pm-A_\pm u_\pm=f_\pm\quad\text{on } D^2_\pm  $$
and 
$\big[ \cB_\nu u \big]_\Gamma\in H^{\frac 1 2+r}(\Gamma_2)$.
Then $u_\pm\in H^{2+r}(D^1_\pm)$.
\end{proposition}

\begin{proposition}[\cite{trudinger}, Theorem 8.10]
\label{prop:GT}
Assume  $(\mathbf{E})$.

Let $r\in \N$. Assume that the coefficients $(a_{\pm})_{ij}$ belong to $C_b^{r+1}(\bar{D}_\pm ; \R)$. Assume $\Gamma$ is bounded.

Let $f_\pm\in H^r(D_\pm)$. Let $u\in H^1(\R^d)$ satisfying 
$$u_\pm-A_\pm u_\pm=f_\pm\quad\text{on } D_\pm.  $$

Let $D'_\pm \subset D_\pm$ open subsets with $\overline{D'_\pm} \subset D_\pm$ and denote $d'_\pm=\mathrm{dist}(D'_\pm,\Gamma)$. 

We have that $u_\pm\in H^{r+2}(D'_\pm)$, with
$$
||u_\pm||_{H^{r+2}(D'_\pm)}\leq C_\pm\big( ||u_\pm||_{H^1(D_\pm)}+||f||_{H^r(D_\pm)}   \big),$$
where the constant $C_\pm$ depends on $d,\lambda, d'_\pm$ and 
$$\max_{1\leq i,j\leq d}\max_{|\alpha|\leq r+1}\sup_{x\in D_\pm}|\partial^\alpha (a_{\pm})_{ij}(x)|.$$
\end{proposition}

\begin{proof}
In \cite{trudinger} this result is asserted with the assumption that $\overline{D'_\pm} \subset D_\pm$, with 
$\overline{D'_\pm}$ compact. 
So that for the interior (bounded) domain $D_+$ the result is immediate.  On the unbounded domain $D_-$
we claim that the same result holds for non compact $\overline{D'_-}$, as in fact only
the distance $d'_-=\mathrm{dist}(D'_-,\Gamma)$ plays a role in the proof. 
\end{proof}

Thus, covering $\Gamma$ with open balls in order to use the local result of Proposition~\ref{prop:reg-sol-faible}, and combining with the global result of 
Proposition \ref{prop:GT}, it is possible to show the following theorem, that will be used extensively in the sequel.

\begin{theorem}
\label{thm:reg-sol-faible}

Assume $(\mathbf{E})$.

Let $r\in \N$. Assume that the coefficients $(a_{\pm})_{ij}$ belong to $C_b^{r+1}(\bar{D}_\pm ; \R)$. Assume $\Gamma$ is bounded and of class 
$C^{r+2}$.

Let $f_\pm\in H^r(D_\pm)$. Let $u\in H^1(\R^d)$ satisfying 
$$u_\pm-A_\pm u_\pm=f_\pm\quad\text{on } D_\pm  $$
and 
$\big[ \cB_\nu u \big]_\Gamma\in H^{\frac 1 2+r}(\Gamma)$.
Then $u_\pm\in H^{2+r}(D_\pm)$.
\end{theorem}

%

\subsubsection{Existence of a weak solution to the resolvent equation and immediate properties of functions in $\cD(A^k)$, $k\geq 1$}

We have the next result.

\begin{proposition}
\label{prop:sol-faible-ell}
Assume  $(\mathbf{E})$.
Let $f\in L^2(\R^d)$. Then \eqref{eq:resol} has a unique solution in $\cD(A)$.
\end{proposition}

\begin{proof}
Let us note that the symmetric bilinear form on $H^1(\R^d)$
\begin{equation*}
(u,v)\mapsto\langle u,v\rangle_{L^2(\R^d)}+\cE(u,v)
\end{equation*}
is continuous and, thanks to Assumption {\bf (E)}, coercive.
Thus the Lax-Milgram theorem (\cite{brezis} Corollary V.8) immediately asserts the existence of a unique
 $u\in H^1(\R^d)$ such that
 \begin{equation*}
 \forall v\in H^1(\R^d),\quad \langle u,v\rangle_{L^2(\R^d)}+\cE(u,v)=\langle f,v\rangle_{L^2(\R^d)}.
 \end{equation*}
In other words we have for any $\varphi\in C^\infty_c(\R^d;\R)$,
$$
\cE(u,\varphi)=-\langle \sum_{i,j=1}^dD_i(a_{ij}D_ju),\varphi\rangle_{H^{-1}(\R^d),H^1(\R^d)}=\langle (f-u),\varphi\rangle_{L^2(\R^d)}.$$
Hence the distribution $\sum_{i,j=1}^dD_i(a_{ij}D_ju)$ belongs to $L^2(\R^d)$, and thus~$u\in\cD(A)$. Finally, from the above relations we deduce
$$
 \forall v\in H^1(\R^d),\quad \langle u-Au,v\rangle_{L^2(\R^d)}=\langle f,v\rangle_{L^2(\R^d)},
$$
which implies \eqref{eq:resol}.
\end{proof}

The proposition below gives properties of functions belonging to $\cD(A)$. It indicates that the solution $u\in\cD(A)$ of \eqref{eq:resol}
encountered in Proposition \ref{prop:sol-faible-ell} satisfies a continuity property and a transmission condition in a weak sense at the interface.


\begin{proposition}
\label{prop-sautzero}
Let $u\in \cD(A)$. Then $[u]_\Gamma=\big[ \cB_\nu u \big]_\Gamma=0$ a.e. on $\Gamma$.
\end{proposition}

\begin{proof}
Let $u\in\cD(A)$. As $u\in H^1(\R^d)$ one gets by Lemma \ref{lem:exo-mclean} that $[u]_\Gamma=0$ a.e. on $\Gamma$. Set now $f=u-Au\in L^2(\R^d)$. According to Equation \eqref{eq:lienAE} we have
\begin{equation}
 \label{eq:resol-faible}
 \forall v\in H^1(\R^d),\quad \langle u,v\rangle_{L^2(\R^d)}+\cE(u,v)=\langle f,v\rangle_{L^2(\R^d)},
 \end{equation}
 and this in true in particular for any $v\in C^\infty_c(D_+;\R)$. 
 But using Lemma \ref{lem:restri} one has for any $v\in C^\infty_c(D_+;\R)$, that
 $$\cE(u,v)=\sum_{i,j=1}^d\int_{D_+}(a_+)_{ij}(D_ju)_+D_iv=\sum_{i,j=1}^d\int_{D_+}(a_+)_{ij}(D_ju_+)D_iv=\cE_+(u_+,v).$$
 Using now \eqref{eq:lienAEp} we see that 
 $ \langle u_+-A_+u_+,v\rangle_{L^2(D_+)}=\langle f_+,v\rangle_{L^2(D_+)}$ for any $v\in C^\infty_c(D_+;\R)$.
 Proceeding in the same manner on $D_-$ we finally see that $u_\pm-A_\pm u_\pm=f_\pm$ on $D_\pm$.

 Note that by construction $f=f_++f_-$. Using now Lemma \ref{lem:green3}, and comparing \eqref{eq:green2s} and \eqref{eq:resol-faible}, one gets 
 $\big( \big[ \cB_\nu u \big]_\Gamma, \gamma(v) \big)_\Gamma=0$ for any $v\in H^1(\R^d)$. Using the fact that the trace operator is surjective, this implies that  
 $\big( \big[ \cB_\nu u \big]_\Gamma, w \big)_\Gamma=0$ for any $w\in H^{1/2}(\Gamma)$, which completes the proof.
\end{proof}

Thanks to Theorem \ref{thm:reg-sol-faible} we can get as a corollary the following result concerning the iterated domains $\cD(A^k)$, $k\in\N^*$.

\begin{corollary}
\label{cor:regDA} Assume  $(\mathbf{E})$.
Let $k\in \N^*$ and $u\in \cD(A^k)$. Assume that the coefficients $(a_\pm)_{ij}\in C_b^{2k-1}(D_\pm)$ and that 
$\Gamma$ is bounded and of class $C^{2k}$. Then $u_\pm\in H^{2k}(D_\pm)$.
\end{corollary}

\begin{proof}
The proof proceeds by induction on $k$. 

\vspace{0.1cm}

Let $u\in \cD(A)$ (case $k=1$). We have $\big[ \cB_\nu u \big]_\Gamma=0$, according to Proposition \ref{prop-sautzero}. Thus in particular
$\big[ \cB_\nu u \big]_\Gamma\in H^{\frac 1 2}(\Gamma)$. As in the proof of Proposition \ref{prop-sautzero} we set $f=u-Au$ and notice that we have
$u_\pm-A_\pm u_\pm=f_\pm$ on $D_\pm$, with $f_\pm\in L^2(D_\pm)$.

Using Theorem \ref{thm:reg-sol-faible} - remember that $u$ is in $H^1(\R^d)$, $(a_\pm)_{ij}\in C_b^{1}(\bar{D}_\pm;\R)$ and~$\Gamma$ is bounded of class~$C^2$~- we get that $u_\pm\in H^2(D_\pm)$.

\vspace{0.1cm}
Suppose now that the result is true at rank $k-1$ we prove its validity at rank $k$ ($k\geq 2$). Let~$u\in \cD(A^k)$. As $u\in \cD(A)$ we have
$\big[ \cB_\nu u \big]_\Gamma=0\in H^{2k-\frac 3 2}(\Gamma)$. As  $Au\in \cD(A^{k-1})$ the quantity $u-Au=:f$ satisfies
$f_\pm\in H^{2k-2}(D_\pm)$, using the induction hypothesis. But as we have $u_\pm-A_\pm u_\pm=f_\pm$ on $D_\pm$, one may use again the smoothness of $(a_\pm)_{ij}$ and $\Gamma$ and Theorem \ref{thm:reg-sol-faible} in order to conclude that $u_\pm\in H^{2k}(D_\pm)$. 
\end{proof}

\subsection{The solution of the parabolic problem $(\mathcal{P}_\mathrm{T})$}

\subsubsection{Application of the Hille-Yosida theorem}
We now use the Hille-Yosida theorem (\cite{brezis} Theorems VII.4 and VII.5) in order to prove the following proposition. Note that in Equation 
\eqref{eq:dudtfaible} below, the time derivative is understood in the strong sense, while the space derivatives are understood in the weak sense.
Besides, by convention $\cD(A^0)=L^2(\R^d)$.

\begin{proposition}
\label{prop:hille}
Assume  $(\mathbf{E})$.
Let $u_0\in \cD(A)$.
Then there exists a unique function 
$$
u\in C^{1}\big([0,T];\,  L^2(\R^d) \big)\cap C\big([0,T];\, \cD(A)  \big)
$$
satisfying
\begin{equation}
\label{eq:dudtfaible}
\frac{\mathrm{d}u}{\mathrm{d}t}=Au,\quad\quad u(0)=u_0.
\end{equation}
Furthermore, let $u_0\in \cD(A^k)$, $k\geq 2$. Then, $$
u\in C^{k-j}\big([0,T];\, \cD(A^j)  \big),\quad 0\leq j\leq k.$$
\end{proposition}

\begin{proof}
According to \cite{brezis} it suffices to check that $(-A,\cD(A))$ is maximal monotone. 
But thanks to Assumption {\bf (E)} we immediately see that
$\langle -Av,v\rangle_{L^2(\R^d)}=\cE(v,v)\geq 0$, for any $v\in\cD(A)$, and thanks to Proposition \ref{prop:sol-faible-ell} we have that for any $f\in L^2(\R^d)$ there exists $u\in\cD(A)$ solving \eqref{eq:resol}.
\end{proof}

Using now Proposition \ref{prop-sautzero}, Corollary \ref{cor:regDA} and Proposition \ref{prop:hille} together with some Sobolev embedding theorems, we show Theorem \ref{thm-existe-sol-classique}.

\subsubsection{Proof of Theorem \ref{thm-existe-sol-classique}}
\begin{proof}
Assume $d$ is even.
Apply the result of Proposition \ref{prop:hille} with $k=k_0 = \lfloor \frac{d}{4}\rfloor + 2$ and consider $u$ solution of \eqref{eq:dudtfaible}.
We have that 
$$
u\in C^1\pare{[0,T];\cD\pare{A^{k_0-1}}}
$$
with $k_0-1 = \lfloor\frac{d}{4}\rfloor + 1$.
Using the result of Corollary \ref{cor:regDA} and combining Corollary IX.13 p. 168 with Theorem IX.7 p. 157 in \cite{brezis}, we see that for any $t\in [0,T]$
\begin{equation}
u_\pm(t,.)\in H^{4 + 2\lfloor\frac{d}{4}\rfloor}(D_{\pm})\subset H^{2 + \frac{d}{2}}(D_{\pm})\xhookrightarrow{} C^{2}(\bar{D}_{\pm}).
 \end{equation}
 
 Assume now that $d$ is odd.
 Apply the result of Proposition \ref{prop:hille} with $k=k_0 = \lfloor\frac{3}{2} + \frac{\lfloor d/2\rfloor}{2}\rfloor+2$ and consider $u$ solution of \eqref{eq:dudtfaible}. 
We have that 
$$
u\in C^1\pare{[0,T];\cD\pare{A^{k_0-1}}}
$$
with $k_0-1 = \lfloor\frac{3}{2} + \frac{\lfloor d/2\rfloor}{2}\rfloor + 1$.
Using the result of Corollary \ref{cor:regDA} and combining Corollary IX.13 p. 168 with Theorem IX.7 p. 157 in \cite{brezis}, we see that for any $t\in [0,T]$
\begin{equation}
u_\pm(t,.)\in H^{2+ 2\lfloor\frac{3}{2} + \frac{\lfloor d/2\rfloor}{2}\rfloor}(D_{\pm})\xhookrightarrow{} C^{2}(\bar{D}_{\pm})
 \end{equation}
 since 
 \begin{align*}
 \lfloor 2 + 2\lfloor\frac{3}{2} + \frac{\lfloor d/2\rfloor}{2}\rfloor - \frac{d}{2}\rfloor \geq \lfloor 2 + 2\pare{\frac{1}{2} + \frac{\lfloor d/2\rfloor}{2}} - \frac{d}{2}\rfloor\geq \lfloor3 + \lfloor\frac{d}{2}\rfloor - \frac{d}{2}\rfloor \geq 2.
 \end{align*}

Let us now show that $u$ solution of \eqref{eq:dudtfaible} (for the corresponding $k_0$) is a classical solution of $(\mathcal{P}_\mathrm{T})$.

First, it is clear that $\cL u$ coincides with $Au$ on any bounded part of $D_\pm$ (the derivatives in the distributional sense coincide with the classical derivatives thanks to the established smoothness of $u$). This shows the first line of $(\mathcal{P}_\mathrm{T})$.

Second, as for any $t\in [0,T]$ the function $u(t,.)$ belongs to $\cD(A)$, we have using the result of Proposition \ref{prop-sautzero} that
\begin{equation}
\label{eq:sautzero-t}
\left [u(t,.)\right ]_\Gamma = 0 \;\;\;\text{a.e. on }\Gamma;\;\;\;\;\left [\cB_\nu u(t,.)\right ]_\Gamma = 0\;\;\;\text{a.e. on }\Gamma.
\end{equation}
Note that $u(t,.) \in \cD(A)$ implies that $u_\pm(t,.)$ are in $H^{2}(D_\pm)$. So that the second part of \eqref{eq:sautzero-t} reads
$$
\nu^\ast \pare{\gamma\pare{a_+\nabla u_+(t,.)} - \gamma\pare{a_-\nabla u_-(t,.)}} = 0\;\;\;\text{a.e.}
$$
But as $\pare{a_\pm \nabla u_\pm}\in C^1(\bar{D}_\pm; \R^d)$, we get
$$
\langle\,(a_+\nabla_x u_+(t,.))(y) - (a_-\nabla_x u_-(t,.))(y), \nu(y)\,\rangle = 0
$$
for almost every $y\in \Gamma$, and consequently for every $y\in \Gamma$ by continuity.
The same argument applies to the first part of \eqref{eq:sautzero-t} and the second and third lines of $(\mathcal{P}_\mathrm{T})$ are satisfied. Note that the constructed solution satisfies $u(t,.)\in C\pare{\R^d}$ for any time $t\in [0,T]$.

Now let $k\geq k_0$. For $\lceil\frac{d}{4}\rceil\leq j\leq k$, we have $2j-\frac{d}{2}>0$. Thus, for $v\in \cD(A^j)$ we have from Corollary~\ref{cor:regDA},
$$
v_\pm\in H^{2j}(D_\pm)\xhookrightarrow{} C^{n(j)}(\bar{D}_{\pm})
$$
with $n(j) = \lfloor2j-\frac{d}{2}\rfloor$. Using again the result of Proposition \ref{prop:hille}, we get the announced result.
 \end{proof}

\subsection{Conclusion and consequences: boundedness of the partial derivatives}

Going a bit further in the analysis, and using additional Sobolev embedding arguments, we can state the following result.

\begin{proposition}
\label{prop:der-born}
Assume  $(\mathbf{E})$.
Let $p,q\in\N$ with $p+\lfloor q/2\rfloor\geq 2$. Let $m=\lceil \frac q 2 +\frac d 4\rceil$, $m'=m+1$ and $k=m'+p$. 
Assume that the coefficients $a_{ij}$ satisfy $(a_\pm)_{ij}\in C_b^{2m'-1}(\bar{D}_\pm)$, that $\Gamma$ is bounded and of class
$C^{2m'}$, and that $u_0\in\cD(A^k)$.

Then the classical solution 
$u(t,x)$ of $(\cP_{\mathrm{T}})$ constructed in Theorem \ref{thm-existe-sol-classique} satisfies
$$
u\in C^p([0,T]\,;\,C^q_b(\bar{D}_+)\cap C^q_b(\bar{D}_-)).$$
\end{proposition}

\begin{proof}
First, notice that it is easy to check that $k$ is greater than the $k_0$ defined in Theorem \ref{thm-existe-sol-classique}, so that it makes sense speaking of the classical solution of $(\cP_{\mathrm{T}})$, for $u_0\in\cD(A^k)$. 

This solution is constructed in the same way as in Theorem \ref{thm-existe-sol-classique}, in particular by the mean of Proposition \ref{prop:hille}. So that one can assert that
$$
u\in C^p([0,T]\,;\,\cD(A^{m'})).
$$
It remains to check that if $v\in\cD(A^{m'})$, then $v_\pm\in C^q_b(\bar{D}_\pm)$. First, note that $m\geq \lceil \frac d 4\rceil$, and that one may easily check
$$
\lfloor 2m-\frac d 2 \rfloor \geq q$$
(using in particular $\lceil 2a \rceil\leq 2\lceil a\rceil$). So that if $v\in \cD(A^{m'})\subset \cD(A^m)$, we have, as
for the second part of Theorem \ref{thm-existe-sol-classique},
$$
v_\pm\in H^{2m}(D_\pm)\xhookrightarrow{} C^{\lfloor 2m-\frac d 2 \rfloor}(\bar{D}_{\pm})\subset C^q(\bar{D}_{\pm}).
$$
We claim that for any multi-index $\alpha$, $|\alpha|\leq q$, the partial derivatives $\partial^\alpha v_\pm$ are bounded.
Indeed, using again Corollary \ref{cor:regDA}, we get
$$
v_\pm\in H^{2m'}(D_\pm),$$
so that for $\alpha$, $|\alpha|\leq q$,
$$\partial^\alpha v_\pm \in H^{2\lceil \frac q 2 +\frac d 4\rceil-q+2}(D_\pm)\subset H^{\frac d 2+2}(D_\pm)\xhookrightarrow{} L^\infty(D_\pm).$$
Here we have used the fact $\frac 1 2-\frac 1 2-\frac 2 d<0$, so that one can use the third embedding result of Corollary IX.13 in \cite{brezis} (and again Theorem IX.7 for the projection argument). The result is proved.
\end{proof}

From the above proposition we get the following control on the partial derivatives of the solution to~$(\cP_{\mathrm{T}})$.

\begin{corollary}
\label{cor:control-der}
In the context of Proposition \ref{prop:der-born} we have
$$
\sup_{t\in[0,T]}\sup_{x\in \bar{D}_\pm}|\partial^j_t\partial^\alpha u_\pm(t,x)|<\infty
$$
for any $j\leq p$ and any multi-index $\alpha$, with $|\alpha|\leq q$.
\end{corollary}

\begin{proof}
By Proposition \ref{prop:der-born} any of the considered partial derivatives of $u_\pm$ belongs to the space
$$
C([0,T]\,;\,C_b(\bar{D}_\pm)).$$
Let for example $v\in C([0,T]\,;\,C_b(\bar{D}_+))$. We prove the continuity of the map $t\mapsto \sup_{x\in \bar{D}_+}|v(t,x)|$, $t\in [0,T]$. Let 
$t_0\in[0,T]$.
Using the reverse triangle inequality we get for any $t\neq t_0$,
$$
\big|  \sup_{x\in \bar{D}_+} |v(t,x)|    -  \sup_{x\in \bar{D}_+}|v(t_0,x)|   \,\big|\leq \sup_{x\in \bar{D}_+} \big| v(t,x)-v(t_0,x)   \big|,$$
and we get the continuity at $t_0$, as $v$ is continuous from $[0,T]$ to $C_b(\bar{D}_+)$ (equipped with the supreme norm). Thus the desired continuity is proved, and from this we can assert that
$$
\sup_{t\in[0,T]}\sup_{x\in \bar{D}_+}|v(t,x)|=\sup_{x\in \bar{D}_+}|v(t^*,x)|$$
for some $t^*\in[0,T]$. As $v(t^*,\cdot)\in C_b(\bar{D}_+)$ we have that
$$
\sup_{t\in[0,T]}\sup_{x\in \bar{D}_+}|v(t,x)|<\infty.
$$
The result is proved.
\end{proof}

In the analysis of the convergence of our Euler scheme, we will use the above corollary with $p$ up to $2$ and~$q$ up to $4$.



\section{Euler scheme}
\label{sec:euler}

\subsection{Recalls on the projection and the distance to the transmission boundary and further notations and premiminaries}

In this subsection we adopt the notations from \cite{Bossy-et-al-2004}. We have the following set of geometric results.

\begin{proposition}[\cite{Bossy-et-al-2004}, Proposition 1; see also \cite{gobet_2001}]
\label{prop:gobet}
Assume  $\Gamma$ is bounded and of class $C^5$.  Assume  $(\mathbf{E})$. Assume that the coefficients $a_{ij}$ satisfy $(a_\pm)_{ij}\in C_b^{4}(\bar{D}_\pm)$.   

There is constant $R>0$ such that:
\begin{enumerate}
\item 
\begin{enumerate}
\item for any $x\in V^-_{\Gamma}(R)$, there are unique $s=\pi_{\Gamma}^{\gamma_+}(x)\in \Gamma$ and $F^{\gamma_+}(x)\leq 0$ such that~:~
\begin{equation}
x = \pi_{\Gamma}^{\gamma_+}(x) + F^{\gamma_+}(x){\gamma_+}(\pi_{\Gamma}^{\gamma_+}(x))\;;
\end{equation} 
\item for any $x\in V^+_{\Gamma}(R)$, there are unique $s=\pi_{\Gamma}^{\gamma_-}(x)\in \Gamma$ and $F^{\gamma_-}(x)\leq 0$ such that~:~
\begin{equation}
x = \pi_{\Gamma}^{\gamma_-}(x) + F^{\gamma_-}(x){\gamma_-}(\pi_{\Gamma}^{\gamma_-}(x))\;;
\end{equation} 
\end{enumerate}
\item 
\begin{enumerate}
\item the function $x\mapsto \pi_{\Gamma}^{\gamma_+}(x)$ is called the projection of $x$ on $\Gamma$ parallel to $\gamma_+$~:~this is a $C^4$ function on $V^-_{\Gamma}(R)\;;$
\item the function $x\mapsto \pi_{\Gamma}^{\gamma_-}(x)$ is called the projection of $x$ on $\Gamma$ parallel to $\gamma_-$~:~this is a $C^4$ function on $V^+_{\Gamma}(R)\;;$
\end{enumerate}
\item Let us set $\tilde{F}^{\gamma_\pm}(x) = F^{\gamma_{\pm}}(x)|\gamma_{\pm}\pare{\pi_{\Gamma}^{\gamma_\pm}(x)}|$ the normalized version of $F^{\gamma_{\pm}}$ corresponding to the unit vector field $\tilde{\gamma}_{\pm}~:~x\mapsto \frac{\gamma_{\pm}(x)}{|\gamma_{\pm}(x)|}$. 
\begin{enumerate}
\item the functions $x\mapsto \tilde{F}^{\gamma_{\pm}}(x)$ are called the algebraic distance of $x$ to $\Gamma$ parallel to $\gamma_\pm$ (to $\tilde{\gamma}_{\pm}$)~:~these are $C^4$ functions on $V^\mp_{\Gamma}(R)$. One has $F^{\gamma_+}, \tilde{F}^{\gamma_+}\leq 0$ on $V^-_{\Gamma}(R)$ and $F^{\gamma_-}, \tilde{F}^{\gamma_-}\leq 0$ on $V^+_{\Gamma}(R)$.
\item It is possible to extend $F^{\gamma_+}$, $\tilde{F}^{\gamma_+}$ and $F^{\gamma_-}, \tilde{F}^{\gamma_-}$ to $C_b^4(\R^d,\R)$ functions, with the conditions $F^{\gamma_\pm}, \tilde{F}^{\gamma_\pm}>0$ on $D_\pm$ and $F^{\gamma_\pm}, \tilde{F}^{\gamma_\pm}<0$ on $D_\mp$. 
\end{enumerate}
\item The above extensions for $\tilde{F}^{\gamma_\pm}$ and $F^{\nu}$ can be performed in a way such that the functions $\tilde{F}^{\gamma_{\pm}}$ and $F^{\nu}$ are equivalent in the sense that for all $x\in \R^d$,
\begin{equation}
\label{eq:equiv-dist-tildeF}
\frac{1}{c_1}d(x,\Gamma) = \frac{1}{c_1}\left |F^{\nu}(x)\right |\leq \left  |\tilde{F}^{\gamma_{\pm}}(x)\right |\leq c_1 \left |F^{\nu}(x)\right | = c_1 d(x,\Gamma)
\end{equation}
for some constant $c_1>1$.
\item For $x\in \Gamma$,
\begin{equation}
\label{eq:gradient-tildeF}
\nabla \tilde{F}^{\gamma_\pm}(x)  = \frac{\nu^\ast}{\langle \nu, \tilde{\gamma}_{\pm}\rangle }(x).
\end{equation}
\end{enumerate}
\end{proposition}

\begin{remark}
Under the assumptions of Proposition \ref{prop:gobet} we have that the vector fields $\gamma_\pm(x), x\in\Gamma$, are of class $C^4$, and we have \eqref{eq:controle-conom}. Thus we are indeed under the assumptions of
 Proposition 1 in \cite{Bossy-et-al-2004}.
\end{remark}

We sometimes use the notation $\nu(x)$ or $\gamma_{\pm}(x)$ even if $x\notin \Gamma$. For $x \in V^{\pm}_\Gamma(R)$, we set
$\nu(x) = \nu(\pi_\Gamma^{\gamma_\pm}(x))$ and $\gamma_{\pm}(x) = \gamma_{\pm}(\pi_\Gamma^{\gamma_\pm}(x))$ and for $x\notin
V_\Gamma^{{\pm}}(R)$, arbitrary values are given.

\vspace{0.2cm}
Note that if $u$ is a classical solution to the transmission parabolic problem $(\cP_{\mathrm{T}})$ defined in Section \ref{sec:EDP}, the transmission condition $(\star)$ can be expressed as 
\begin{equation}
\label{eq:trans-gamma}
\langle \gamma_+(y)\,,\,\nabla_xu_+(t,y)\rangle=-\langle \gamma_-(y)\,,\nabla_xu_-(t,y)\rangle,\quad \forall (t,y)\in (0,T]\times \Gamma.
\end{equation}
This in fact will be the crux of our approach (see Subsubsection \ref{ssec:cancel-tc}).

\vspace{0.2cm}

In the sequel, we will need the following result.
\begin{proposition}
\label{prop:geometrique}
Assume  $\Gamma$ is bounded and of class $C^5$.  Assume  $(\mathbf{E})$. Assume that the coefficients $a_{ij}$ satisfy $(a_\pm)_{ij}\in C_b^{4}(\bar{D}_\pm)$.   

Let $\hat{x}\in V_{\Gamma}^{\mp}(R)$ and $\overline{x}\in V_{\Gamma}^{\mp}(R)$ be linked by the following relation~:~
\begin{equation}
\label{eq:relation-x-xchapeau}
\overline{x} = \pi_{\Gamma}^{\gamma_{\pm}}(\hat{x}) - F^{\gamma_{\pm}}(\hat{x})\gamma_{\mp}(\pi_{\Gamma}^{\gamma_{\pm}}(\hat{x})).
\end{equation}
Then, there exists $c_2>1$ such that
\begin{equation}
\label{eq:equiv-dist-transformation}
\frac{1}{c_2}\, d(\overline{x},\Gamma)\leq d(\hat{x},\Gamma)\leq {c_2}\, d(\overline{x},\Gamma).
\end{equation}
\end{proposition}
\begin{proof}
Without loss of generality, assume for example that $\overline{x}\in V_\Gamma^{-}(R)$ and $\hat{x}\in V_\Gamma^{-}(R)$ are related by \eqref{eq:relation-x-xchapeau}. 
Then we have
\begin{equation}
\label{eq:relation-x-xchapeau-cote}
\overline{x} - \pi_{\Gamma}^{\gamma_{+}}(\hat{x}) = -F^{\gamma_+}(\hat{x})\gamma_{-}(\pi^{\gamma_{+}}_{\Gamma}(\hat{x})).
\end{equation}
and by uniqueness of the projection $\pi_\Gamma^{-\gamma_-}(\overline{x})$, we see that $\pi^{\gamma+}_{\Gamma}(\hat{x}) = \pi_\Gamma^{-\gamma_-}(\overline{x})$ (note that $F^{-\gamma_-}(\overline{x})=F^{\gamma_+}(\hat{x})$).

We deduce that
\begin{align*}
\frac{1}{c_1} d(\overline{x},\Gamma)\leq |\tilde{F}^{-\gamma_-}(\overline{x})|=
|F^{-\gamma_-}(\overline{x})|\times |\gamma_-(\pi_\Gamma^{-\gamma_-}(\overline{x}))|
=|\overline{x} - \pi_\Gamma^{\gamma_+}(\hat{x})| \\
\\
 = |\overline{x} - \pi_\Gamma^{-\gamma_-}(\overline{x})|   \leq c_1 d(\overline{x},\Gamma)\\
\end{align*}
due to the same kind of relation as \eqref{eq:equiv-dist-tildeF}, but written for $-\gamma_{-}$ instead of $\gamma_{-}$.
Returning back to \eqref{eq:relation-x-xchapeau-cote}, we see that 
$$
\frac{1}{c_1} d(\overline{x},\Gamma)\leq 
|F^{\gamma_+}(\hat{x})|\times |\gamma_{-}(\pi_\Gamma(\hat{x}))|= 
|\tilde{F}^{\gamma_+}(\hat{x})|\,\frac{|\gamma_{-}(\pi_\Gamma(\hat{x}))|}{|\gamma_{+}(\pi_\Gamma(\hat{x}))|}  \leq c_1 d(\overline{x},\Gamma).
$$
So that in view of \eqref{eq:equiv-dist-tildeF} written for $\hat{x}$ and $\gamma_+$,
$$
\frac{1}{c_1^2} \frac{|\gamma_{+}(\pi_\Gamma(\hat{x}))|}{|\gamma_{-}(\pi_\Gamma(\hat{x}))|} d(\overline{x},\Gamma)\leq d(\hat{x},\Gamma) \leq c_1^2 \frac{|\gamma_{+}(\pi_\Gamma(\hat{x}))|}{|\gamma_{-}(\pi_\Gamma(\hat{x}))|} d(\overline{x},\Gamma).
$$
But using   \eqref{eq:ell} and \eqref{eq:decomp-ortho}, it easy to see that for any $z\in \Gamma$,
\begin{equation*}
\frac{\lambda^2}{\Lambda^2d^2}\leq \frac{|\gamma_{+}(z)|^2}{|\gamma_{-}(z)|^2} \leq \frac{\Lambda^2d^2}{\lambda^2}
\end{equation*}
from which we deduce the result of the proposition.
\end{proof}

\subsection{Our transformed Euler scheme} 
\label{ssec:euler}
We are now in position to introduce our transformed Euler scheme.

Let us denote from now on $\triangle t=h_n=\frac T n$ the time step (where $n\in \N^\ast$) and fix a starting point $x_0\in \R^d$.

The time grid is given by $(t^n_k)_{k=0}^n$ with $t^n_k=\frac {Tk}{n}$ for $0\leq k\leq n$.

We denote by $(\Delta W_{k+1})_{k=0}^n$ the i.i.d. sequence of Brownian increments constructed on $(\Omega, {\cal F}, \P^{x_0})$ and defined by 
$$\Delta W_{k+1}=W_{t_{k+1}}-W_{t_k}, \quad \forall\,\,0\leq k\leq n.$$

Recall that $\sigma:\R^d\to\R^{d\times d}$ stands for a matrix valued coefficient satisfying
$$
\sigma\sigma^*(x)=2a(x),\quad\forall x\in D.$$

Set $(\partial a(x))_j = {\rm div}(x\mapsto (a_{1j}(x), \dots, a_{nj}(x))).$

\vspace{0.8cm}
Our stochastic numerical scheme $\pare{\overline{X}^n_{t_k}}_{k=0}^n$ is defined as follows (we omit the superscript $n$) 
$$\overline{X}_{0} = x_0$$ and for $t\in (t_k, t_{k+1}]$, we set
{\small
\begin{equation}
\label{eq-scheme}
\left \{ 
\begin{array}{ll}
\hat{X}_{t} =\overline{X}_{t_{k}}+ \sigma(\overline{X}_{t_k})(W_t - W_{t_k}) + \partial a(\overline{X}_{t_k})(t-t_k)  &\!\!\!\!\!\text{ (standard Euler incrementation)}\\
\\
\overline{X}_{t_{k+1}} = \hat{X}_{t_{k+1}}
&\!\!\!\!\!\!\!\!\!\!\!\text{ if }\pare{\overline{X}_{t_k}\in \overline{D}_+ \text{ and } \hat{X}_{t_{k+1}}\in D_+}\\
&\!\!\!\!\!\!\!\!\!\!\text{ or }\pare{\overline{X}_{t_k}\in \overline{D}_- \text{ and } \hat{X}_{t_{k+1}}\in D_-}~;\\
\\
\overline{X}_{t_{k+1}} =\pi^{\gamma_+}_{\Gamma}(\hat{X}_{t_{k+1}}) - F^{\gamma_{+}}(\hat{X}_{t_{k+1}})\gamma_-(\pi^{\gamma_+}_{\Gamma}(\hat{X}_{t_{k+1}})) &\text{ if }\overline{X}_{t_k}\in \overline{D}_+ \text{ and } \hat{X}_{t_{k+1}}\in D_-~;\\
\\
\overline{X}_{t_{k+1}} =\pi^{\gamma_-}_{\Gamma}(\hat{X}_{t_{k+1}}) - F^{\gamma_{-}}(\hat{X}_{t_{k+1}})\gamma_+(\pi^{\gamma_-}_{\Gamma}(\hat{X}_{t_{k+1}})) &\text{ if }\overline{X}_{t_k}\in \overline{D}_- \text{ and } \hat{X}_{t_{k+1}}\in D_+~.
\end{array}
\right .                             
\end{equation}
}

\begin{figure}
\begin{center}
\includegraphics[width=16cm]{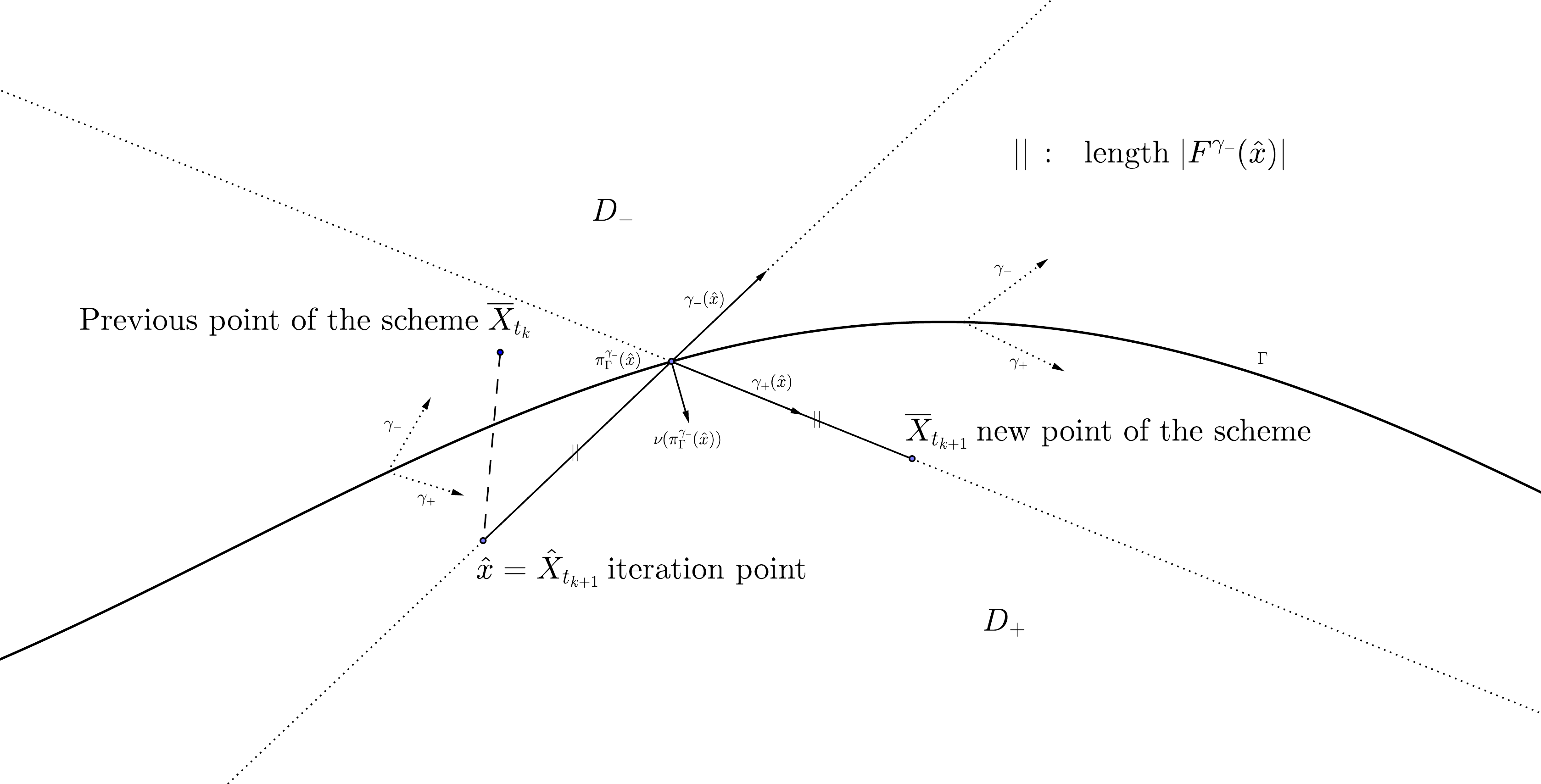}
\caption{Correction of our scheme when the path crosses the boundary $\Gamma$.}
\label{fig:schema}
\end{center}
\end{figure}

\begin{remark}
\label{rem:euler-1D}
When the dimension $d$ is reduced to $1$ (one dimensional problem), the discontinuity surface reduces to a single point (say $0$). In this case and when the coefficient $a = a_+\indi{y>0} + a_-\indi{y<0}$ is constant on both sides of the discontinuity, it is remarkable that our Euler Scheme is exactly the same as the one described in \cite{Martinez-Talay-2012}.

Indeed, in this one-dimensional context, let $\varphi(y) = \pare{a_-\indi{y>0} + a_+\indi{y<0}}y$. Note that $\varphi $ is a bijective map from $\R$ to $\R$. The Euler Scheme constructed in \cite{Martinez-Talay-2012} is then defined by $\overline{X}_0 = x_0$ and for all $k\in \{0,\dots, n\}$,
$$
\overline{X}_{t_k} = \varphi^{-1}\pare{\overline{Y}_{t_k}}
$$ 
where
$\overline{Y}_0 = \varphi(x_0)$ and for all $k\in \{0,\dots, n-1\}$
$$
\overline{Y}_{t_{k+1}} = \overline{Y}_{t_{k}} + \pare{a_-\sigma_+\indi{\overline{Y}_{t_{k}}>0} + a_+\sigma_-\indi{\overline{Y}_{t_{k}}<0}}(W_{t_{k+1}} - W_{t_k})\,; 
$$
(see \cite{Martinez-Talay-2012} for details - please take care that \cite{Martinez-Talay-2012} is written for the right-hand sided local time;
the above computation is valid for the symmetric local time).
For example if $\varphi(\overline{X}_{t_{k}})< 0$ and $\overline{Y}_{t_{k+1}}\geq 0$, we get (because $\varphi^{-1}(0)=0$ and $\varphi^{-1}$ is continuous at $0$ and also because $\overline{X}$ and $\overline{Y}$ share the same sign), 
\begin{align*}
\overline{X}_{t_{k+1}} &= \varphi^{-1}\pare{\varphi(\overline{X}_{t_{k}}) + \pare{a_-\sigma_+\indi{\overline{Y}_{t_{k}}>0} + a_+\sigma_-\indi{\overline{Y}_{t_{k}}<0}}\triangle W_{k}^{k+1}}\\
&=\overline{X}_{t_{k}} + \int_{\varphi(\overline{X}_{t_{k}})}^0 \pare{\varphi^{-1}}'(z)dz \\
&\hspace{1cm}+ \int_0^{\varphi(\overline{X}_{t_{k}}) + (\underbrace{a_-\sigma_+\indi{\overline{Y}_{t_{k}}>0}}_{=0} + a_+\sigma_-\indi{\overline{Y}_{t_{k}}<0})\triangle W_{k}^{k+1}} \pare{\varphi^{-1}}'(z)dz\\
&=\overline{X}_{t_{k}} - \varphi(\overline{X}_{t_{k}})\frac{1}{a_+} + \pare{\varphi(\overline{X}_{t_{k}}) + a_+\sigma_-\triangle W_{k}^{k+1}}\frac{1}{a_-}\\
&=\frac{a_+}{a_-}\overline{X}_{t_{k}} + \frac{a_+}{a_-}\sigma_-\triangle W_{k}^{k+1},
\end{align*}
which turns out to be the corresponding case in \eqref{eq-scheme} in this one-dimensional context.
This correspondence is valid in all cases and our transformed Euler Scheme may be viewed as some kind of generalization of the Euler Scheme presented in \cite{Martinez-Talay-2012}.
\end{remark}

\section{Convergence rate of our Euler scheme}
\label{sec:conv-euler}
The purpose of this section is to prove the following result.

\begin{theorem}
\label{thm-conv-schema} Let $0<T<\infty$. Assume  $(\mathbf{E})$. Let
$m'=\lceil 2 +\frac d 4\rceil+1$ and $k=m'+2$.
Assume that the coefficients $a_{ij}$ satisfy $(a_\pm)_{ij}\in C_b^{2m'-1}(\bar{D}_\pm)$ and that $\Gamma$ is of class
$C^{2m'}$.
Let $u_0:\R^d\to\R$ be in the space ${\cal D}(A^k)$. 
Let $u$ be the classical solution of $(\cP_\mathrm{T})$.

We have that for all $n$ large enough, and all $x_0$ in $\R^d$,
\begin{equation}
\label{eq-err-schema}
\Big| u(T,x_0)-\E^{x_0}u_0(\overline{X}^n_T)  \Big|\leq K \sqrt{h_n},
\end{equation}
where the constant $K$ depends on $d$, $\lambda$, $\Lambda$, $u_0$ and $T$.
\end{theorem}

\begin{remark}
\label{rem:cond-reg-u}
In this theorem the assumptions on $a(x)$ and $\Gamma$ involving the integers $m'$ and $k$ are here in order to use
Corollary \ref{cor:control-der}, which ensures that we will have $
\sup_{t\in[0,T],\,x\in \bar{D}_\pm}|\partial^j_t\partial^\alpha u_\pm(t,x)|<\infty
$
for any $j\leq 2$ and any $|\alpha|\leq 4$. This control on the derivatives on $u$ is what we need in order to lead our convergence proof. In fact if there is a way to get this control under weaker assumptions on $a(x)$ and $\Gamma$, this will lead to a convergence theorem stated under these weaker assumptions.
\end{remark}

\subsection{Preliminary results}

\begin{lemma}\label{lemme:Ito-control}(see \cite{Bossy-et-al-2004})
Consider an It\^o process with uniformly bounded coefficients $dU_t = b_t dt + \sigma_t dW_t$ on $(\Omega, {\cal F}, \P^{x_0})$. 
There exist some constants $c>0$ and $K$ (depending on $p\geq 1$, $T$ and the bounds on $\sigma$, $b$) such that, for any stopping times $S$ and $S'$ (with $0\leq S\leq S'\leq \delta\leq T$) and any $\eta\geq 0$,
\begin{align}
\P^{x_0} \croc{\sup_{t\in [S,S']}|U_t - U_s|\geq \eta}&\leq K \exp\pare{-c\frac{\eta^2}{\delta}}~;\\
\E^{x_0}\croc{\sup_{t\in [S,S']}|U_t - U_s|^p}&\leq K \delta^{p/2}.
\end{align}
\end{lemma}

We have when $\overline{X}_{t_k}\in D_+$
\begin{align*}
\overline{X}_{t_{k+1}}=\hat{X}_{t_{k+1}} + \croc{F^{\gamma_+}(\hat{X}_{t_{k+1}})}^{-}\pare{\gamma_+(\pi^{\gamma_+}_\Gamma(\hat{X}_{t_{k+1}})) + \gamma_{-}(\pi^{\gamma_+}_\Gamma(\hat{X}_{t_{k+1}}))}
\end{align*}
and when $X_{t_k}\in D_-$
\begin{align*}
\overline{X}_{t_{k+1}}=\hat{X}_{t_{k+1}} + \croc{F^{\gamma_-}(\hat{X}_{t_{k+1}})}^{-}\pare{\gamma_+(\pi^{\gamma_-}_\Gamma(\hat{X}_{t_{k+1}})) + \gamma_{-}(\pi^{\gamma_-}_\Gamma(\hat{X}_{t_{k+1}}))}
\end{align*}

This shows that $(\overline{X}_{t})_{0\leq t\leq T}$ behaves like a continuous semimartingale on each of the intervals $[t_k,t_{k+1})$. Using Tanaka's formula, we have -- for example for $\overline{X}_{t_k}\in D_+$ -- that for any $t\in [t_k,t_{k+1})$,
\begin{align}
\label{eq:differentielle-schema}
d\overline{X}_t &= d\hat{X}_t + \frac{1}{2}\pare{\gamma_+ + \gamma_-}(\hat{X}_t)dL_t^0(F^{\gamma_+}(\hat{X}))\nonumber\\
&\quad + \croc{F^{\gamma_+}(\hat{X}_{t})}^{-}\Big(\nabla \pare{\gamma_+ + \gamma_-}(\hat{X}_t)d\hat{X}_t\nonumber \\
&\hspace{3cm}+ \frac{1}{2}{\rm Tr}\croc{\mathbf{H}[\gamma_+ + \gamma_-](\hat{X}_{t})a(\overline{X}_{t_k)})}dt\Big)\nonumber\\
&\quad - \indi{F^{\gamma_+}(\hat{X}_t)< 0}\Big[\nabla\pare{\gamma_+ + \gamma_-}(\hat{X}_t)a(\overline{X}_{t_k})\pare{\nabla F^{\gamma_+}(\hat{X}_t)}^\ast dt\nonumber \\
&\hspace{3cm}+ \pare{\gamma_+ + \gamma_-}(\hat{X}_t)\nabla F^{\gamma_+}(\hat{X}_t)d\hat{X}_t\Big]\nonumber\\
&\quad\quad - \indi{F^{\gamma_+}(\hat{X}_t)< 0}\pare{\gamma_+ + \gamma_-}(\hat{X}_t) \frac{1}{2}{\rm Tr}\croc{\mathbf{H}[F^{\gamma_+}](\hat{X}_t)a(\overline{X}_{t_k})}dt.
\end{align}

\begin{lemma}
\label{lemme:control-couronne}
Under the assumptions of Theorem \ref{thm-conv-schema}, for all $c>0$, there exists a constant $K(T)$ such that
\begin{equation}
h_n\;\E^{x_0}\sum_{i=0}^{n-1}\croc{\exp\pare{-c\frac{d^2(\overline{X}^n_{t_i}, \Gamma)}{h_n}}}\leq K(T) \sqrt{h_n}.
\end{equation}
\end{lemma}
\begin{proof}
The idea is to use the occupation times formula. Using successively \eqref{eq:equiv-dist-tildeF} and the inequality \eqref{eq:equiv-dist-transformation} of Proposition \ref{prop:geometrique}, we have
$d\pare{\overline{x}, \Gamma} \geq \frac{1}{c_2}d\pare{\hat{x}, \Gamma} \geq \frac{1}{c_1c_2}|\tilde{F}^{\gamma_{\pm}}(\hat{x})|$ so that
\begin{align}
{\cal A}_{i+1}&:= \E^{x_0}\croc{\exp\pare{-c\frac{d^2(\overline{X}^n_{t_{i+1}}, \Gamma)}{h_n}}}\nonumber\\
&\leq \E^{x_0}\croc{\exp\pare{-c\frac{\left |\tilde{F}^{\gamma_+}(\hat{X}^n_{t_{i+1}})\right |^2}{|c_1c_2|^2 h_n}}\indi{\overline{X}^n_{t_{i+1}}\in D_-}}\nonumber \\
&\hspace{2cm}+ \E^{x_0}\croc{\exp\pare{-c\frac{\left |\tilde{F}^{\gamma_-}(\hat{X}^n_{t_{i+1}})\right |^2}{|c_1c_2|^2 h_n}}\indi{\overline{X}^n_{t_{i+1}}\in D_+}}\nonumber\\
&\quad := {\cal A}_{i+1}^+ + {\cal A}_{i+1}^-.
\end{align}
We concentrate on term ${\cal A}_{i+1}^+$ as both terms are treated in a similar manner.

Set $c'=c/2c_1^2c_2^2>0$ and $g(x) = \exp(-2c'x^2/h)$; it is easy to check that $|g(x)| + \sqrt{h}|g'(x)| + h |g''(x)| \leq K(T)\exp(-c' x^2 /h)$. Hence, for $t\in [t_i, t_{i+1}]$, It\^o's formula yields that
\begin{align*}
&\E^{x_0}\exp\pare{-2c'\frac{\left |\tilde{F}^{\gamma_+}(\hat{X}^n_{t_{i+1}})\right |^2}{h_n}}\\
&\leq K(T) \Big[\E^{x_0}\exp\big(-c'\frac{\left |\tilde{F}^{\gamma_+}(\hat{X}^n_{t})\right |^2}{h_n}\big) \\
&\hspace{2cm}+  \frac{1}{h_n}\int_t^{t_{i+1}}\,ds\,\E^{x_0}\exp\big(-c'\frac{\left |\tilde{F}^{\gamma_+}(\hat{X}^n_{s})\right |^2}{h_n}\big)\Big ].
\end{align*}
We integrate this inequality with respect to $t$ over $[t_i, t_{i+1}]$ to get
\begin{equation}
h_n\,{\cal A}^+_{i+1} \leq K(T)\int_{t_i}^{t_{i+1}} \,ds\,\E^{x_0}\exp\pare{-c'\frac{\left |\tilde{F}^{\gamma_+}(\hat{X}^n_{s})\right |^2}{h_n}}.
\end{equation}
(for possibly some new constant $K(T)$).

Observe that from \eqref{eq:gradient-tildeF},
\begin{align}
\label{eq:minoration-crochet}
d\langle \tilde{F}^{\gamma_+}(\hat{X}^n), \tilde{F}^{\gamma_+}(\hat{X}^n)\rangle_s &= \nabla \tilde{F}^{\gamma_+}(\hat{X}^n_{s})a(\overline{X}^n_{t_i})\croc{\nabla \tilde{F}^{\gamma_+}(\hat{X}^n_{s})}^\ast\,ds\geq
\lambda ds.
\end{align} 
Indeed, using the Cauchy-Schwarz inequality and $|\nu(\hat{x})|= 1$, we have that 
\begin{align*}
\nabla \tilde{F}^{\gamma_+}(\hat{x})a(\overline{x})\croc{\nabla \tilde{F}^{\gamma_+}(\hat{x})}^\ast &= \frac{\nu^\ast(\hat{x}) a(\overline{x}) \nu(\hat{x})}{\langle\nu(\hat{x}), \tilde{\gamma}_+(\hat{x})\rangle^2}= \frac{\langle\nu(\hat{x}), a(\overline{x})\nu(\hat{x})\rangle}{\langle \nu(\hat{x}), \frac{a(\hat{x})\nu(\hat{x})}{|a(\hat{x})\nu(\hat{x})|}\rangle^2}\\
&\geq \frac{\langle\nu(\hat{x}), a(\overline{x})\nu(\hat{x})\rangle}{|\nu(\hat{x})|^2 |a(\hat{x})\nu(\hat{x})|^2}|a(\hat{x})\nu(\hat{x})|^2 = \langle\nu(\hat{x}), a(\overline{x})\nu(\hat{x})\rangle \\
&\geq \lambda
\end{align*}
which justifies \eqref{eq:minoration-crochet}.

It readily follows from the occupation times formula that
\begin{equation}
h_n\,{\cal A}^+_{i+1} \leq K(T)\int_{-R}^R dy \exp\pare{-c'\frac{y^2}{h_n}}\E^{x_0}\croc{\triangle_{i}^{i+1}L^y\pare{\tilde{F}^{\gamma_+}(\hat{X}^n_{.})}}.
\end{equation}

Now,  
\begin{align*}
&\E^{x_0}\croc{L^y_{t_{i+1}}\pare{\tilde{F}^{\gamma_+}(\hat{X}^n_{.})} - L^y_{t_{i}}\pare{\tilde{F}^{\gamma_+}(\hat{X}^n_{.})}}\\
&= 2\E^{x_0}\Big[\pare{\tilde{F}^{\gamma_+}(\hat{X}^n_{t_{i+1}}) - y}^+ - \pare{\tilde{F}^{\gamma_+}(\hat{X}^n_{t_{i}}) - y}^+ \\
&\hspace{3cm}
- \int_{t_i}^{t_{i+1}}\indi{\tilde{F}^{\gamma_+}(\hat{X}^n_{s})\geq y} d\pare{\tilde{F}^{\gamma_+}(\hat{X}^n_{s})}\Big]\\
&\leq  2\E^{x_0}\croc{\pare{\tilde{F}^{\gamma_+}(\hat{X}^n_{t_{i+1}}) - y}^+ - \pare{\tilde{F}^{\gamma_+}(\hat{X}^n_{t_{i}}) - y}^+} + K(T) h_n.
\end{align*}

Therefore, $\sum_{i=0}^{n-1}\E^{x_0}\croc{L^y_{t_{i+1}}\pare{\tilde{F}^{\gamma_+}(\hat{X}^n_{.})} - L^y_{t_{i}}\pare{\tilde{F}^{\gamma_+}(\hat{X}^n_{.})}}\leq K(T)$ uniformly in $|y|\leq R$ since the sum is telescoping. We can thus conclude that $h_n\sum_{i=0}^{n-1}{\cal A}^+_{i+1}\leq K(T) \sqrt{h_n}$. 

The sum $h_n\sum_{i=0}^{n-1}{\cal A}^-_{i+1}$ is treated similarly. The proof of the Lemma is complete.
\end{proof}

\subsection{Error decomposition}

In all the sequel $x_0$ is arbitrarily fixed.

For all $0\leq k\leq n$ set 
$$ \theta^n_k := T-t^n_k. $$

The proof of Theorem~\ref{thm-conv-schema} proceeds as follows (we omit the superscript $n$). Since $u(0,x)=u_0(x)$ for all $x\in\R^d$ and 
$u(T,x_0)=\mathbb{E}^{x_0}u(T,\overline{X}_0)$, 
the discretization error at time $T$ can be decomposed as follows: 
\begin{equation} \label{error-decomposition}
\begin{split}
\epsilon^{x_0}_T &= \left|u(T,x_0)-\espx{u_0\pare{\overline{X}_T}}\right| \\
&= \Big |\sum_{k=0}^{n-1}\espx{u(T-t_k,\overline{X}_{t_k})}
-\espx{u(T-t_{k+1},\overline{X}_{t_{k+1}})}\Big |,
\end{split}
\end{equation}
and thus
\begin{equation} \label{detailed-error-decomposition}
\begin{split}
\epsilon^{x_0}_T &\leq \Big |\sum_{k=0}^{n-1}
\mathbb{E}^{x_0}\left\{u(\theta_k,\overline{X}_{t_k})
-u(\theta_{k+1},\overline{X}_{t_k})\right.\\
&\quad\quad\quad\quad \left.
+u(\theta_{k+1},\overline{X}_{t_k})
-u(\theta_{k+1},\overline{X}_{t_{k+1}}) \right\}\Big |. \\
\end{split}
\end{equation}

The rest of this section is devoted to the analysis of
$$  \left|\sum_{k=0}^{n-1}\mathbb{E}^{x_0}(T_k-S_k)\right|, $$
where the time increment $T_k$ is defined as
\begin{equation} \label{def:Tk}
T_k := u(\theta_k, \overline{X}_{t_k})-u(\theta_{k+1},\overline{X}_{t_k})
\end{equation}
and the space increment is defined as
\begin{equation} \label{def:Sk}
S_k:= u(\theta_{k+1},\overline{X}_{t_{k+1}})-u(\theta_{k+1},\overline{X}_{t_k}).
\end{equation}

\subsection{Estimate for the time increment  $T_k$}
Remember the definition~(\ref{def:Tk}) of $T_k$ and that $\theta_k=T-t_k$. We have
\begin{equation*}
\begin{split}
&\left\{u(\theta_k,\overline{X}_{t_k})
-u(\theta_{k+1},\overline{X}_{t_k})\right\}\indi{\overline{X}_{t_k}\in  D_+} \\
&\quad =
h_n\partial_tu(\theta_{k+1},\overline{Y}_{t_k})\indi{\overline{X}_{t_k}\in D_+}\\
&\hspace{2cm}+h_n^2\int_{[0,1]^2}\partial^2_{tt}u(\theta_{k+1}+\alpha_1\alpha_2h_n, \overline{X}_{t_k})\alpha_1~d\alpha_1d\alpha_2~\indi{\overline{X}_{t_k}\in D_+} \\
&\quad =: T_k^+ + R_k^+.
\end{split}
\end{equation*}

Similarly,
\begin{equation*}
\begin{split}
&\left\{u(\theta_k,\overline{X}_{t_k})
-u(\theta_{k+1},\overline{X}_{t_k})\right\}\indi{\overline{X}_{t_k}\in D_-}
\\
&\quad = h_n\partial_tu(\theta_{k+1},\overline{X}_{t_k})\indi{\overline{X}_{t_k}\in D_-}\\
&\hspace{2cm}+h_n^2\int_{[0,1]^2}\partial^2_{tt}u(\theta_{k+1}+\alpha_1\alpha_2h_n,\overline{X}_{t_k})\alpha_1~d\alpha_1d\alpha_2~\indi{\overline{X}_{t_k}\in D_-} \\
&\quad =: T_k^-+R_k^-.
\end{split}
\end{equation*}
In view of Corollary~\ref{cor:control-der} and Remark \ref{rem:cond-reg-u} we have
$$\mathbb{E}^{x_0}|R_k^++R_k^-| \leq C\,h_n^2. $$

From the preceding we deduce
\begin{equation} \label{ineq:Tk}
\mathbb{E}^{x_0}T_k = \mathbb{E}^{x_0}\partial_tu(\theta_{k+1},\overline{X}_{t_k})h_n
+O(h_n^2).
\end{equation}

\subsection{Expansion of the space increment $S_k$}
Let $S_k$ be defined as in~(\ref{def:Sk}). Set
\begin{align*}
\triangle_{k+1}\overline{X} &:= \overline{X}_{t_{k+1}}-\overline{X}_{t_k}\\
\triangle^\sharp_{k+1}\overline{X} &:=\hat{X}_{t_{k+1}}-\overline{X}_{t_k}.
\end{align*}
and recall that $\triangle_{k+1}W = W_{t_{k+1}}-W_{t_k}$.

\begin{proposition}
\label{prop:moment-schema}
\begin{equation}
\E^{x_0}|(\triangle^\sharp_{k+1}\overline{X})^{\alpha}|\leq C(\alpha) h_n^{|\alpha |/2}.
\end{equation}
\end{proposition}
\begin{proof}
This is a consequence of the result of Lemma \ref{lemme:Ito-control} combined with the fact that
$|(x)^\alpha |\leq |x|^{|\alpha|}$ for any $x\in \R^d$.
\end{proof}

We emphasize that, due to the definition of our stochastic scheme, $\triangle^\sharp_{k+1}\overline{X}$ 
does not coincide with $\overline{X}_{t_{k+1}}-\overline{X}_{t_k}$
when $\overline{X}_{t_{k+1}}$ and $\overline{X}_{t_k}$ do not belong to the same region, which explains the two notations
$\triangle$ and $\triangle^\sharp$. 

We need to introduce the four following events:
\begin{equation}
\begin{cases}
\Omega_k^{++} &:= [\overline{X}_{t_k}\in \overline{D}_{+}~\text{and}~\hat{X}_{t_{k+1}}\in D_+], \\
\Omega_k^{--} &:= [\overline{X}_{t_k}\in D_-~\text{and}~\hat{X}_{t_{k+1}}\in \overline{D}_{-}], \\
\Omega_k^{+-} &:= [\overline{X}_{t_k}\in D_+~\text{and}~\hat{X}_{t_{k+1}}\in \overline{D}_{-}], \\
\Omega_k^{-+} &:= [\overline{X}_{t_k}\in \overline{D}_{-}~\text{and}~\hat{X}_{t_{k+1}}\in D_+] .
\end{cases}
\end{equation}

In view of the definition of our stochastic numerical scheme we have
$$ \text{On}~\Omega_k^{++},~\triangle_{k+1}\overline{X} = \triangle^\sharp_{k+1}\overline{X}. $$
Therefore
\begin{equation*}
\begin{split}
S_k\indi{\Omega_k^{++}}
&= \langle \triangle_{k+1}\overline{X}, \nabla_x
u(\theta_{k+1},\overline{X}_{t_k})\rangle~\indi{\Omega_k^{++}}\\
&+\dfrac{1}{2}\pare{\triangle_{k+1}\overline{X}}^\ast {\rm \bf H}[u](\theta_{k+1},\overline{X}_{t_k})\triangle_{k+1}\overline{X}
~\indi{\Omega_k^{++}}\\
&\quad +\sum_{|{\alpha}|= 3}\frac{1}{\alpha !}(\triangle_{k+1}\overline{X})^\alpha\frac{\partial^{3}u}{\partial x^{\alpha}}(\theta_{k+1},\overline{X}_{t_k})~\indi{\Omega_k^{++}}\\
&\quad + \int_0^1 d\xi \sum_{|{\alpha}|= 4}\frac{(1-\xi)^4}{\alpha !}(\triangle_{k+1}\overline{X})^\alpha\frac{\partial^{4}u}{\partial x^{\alpha}}(\theta_{k+1},\overline{X}_{t_k} + \xi\triangle_{k+1}\overline{X})~\indi{\Omega_k^{++}} \\
&=: S_k^{++1}+S_k^{++2}+S_k^{++3}+S_k^{++4}.
\end{split}
\end{equation*}
Similarly,
\begin{equation*}
\begin{split}
S_k\indi{\Omega_k^{--}} 
&= \langle \triangle_{k+1}\overline{X}, \nabla_x
u(\theta_{k+1},\overline{X}_{t_k})\rangle~\indi{\Omega_k^{--}}\\
&+\dfrac{1}{2}\pare{\triangle_{k+1}\overline{X}}^\ast {\rm \bf H}[u](\theta_{k+1},\overline{X}_{t_k})\triangle_{k+1}\overline{X}
~\indi{\Omega_k^{--}}\\
&\quad +\sum_{|{\alpha}|= 3}\frac{1}{\alpha !}(\triangle_{k+1}\overline{X})^\alpha\frac{\partial^{3}u}{\partial x^{\alpha}}(\theta_{k+1},\overline{X}_{t_k})~\indi{\Omega_k^{--}}\\
&\quad + \int_0^1 d\xi \sum_{|{\alpha}|= 4}\frac{(1-\xi)^4}{\alpha !}(\triangle_{k+1}\overline{X})^\alpha\frac{\partial^{4}u}{\partial x^{\alpha}}(\theta_{k+1},\overline{X}_{t_k} + \xi\triangle_{k+1}\overline{X})~\indi{\Omega_k^{--}} \\
&=: S_k^{--1}+S_k^{--2}+S_k^{--3}+S_k^{--4}.
\end{split}
\end{equation*}

We now use that $\Omega_k^{++}\cup\Omega_k^{--}=\Omega-(\Omega_k^{+-}\cup\Omega_k^{-+})$.
Notice that $\Omega_k^{+-}\cup\Omega_k^{-+}$ belongs to the $\sigma$-field generated by $(W_t)$
up to time $t_{k+1}$. In view of the first line of \eqref{eq-scheme} and the fact that ${\mathbb E}^{{\cal F}_{t_k}}\Delta W_{k+1} = 0$, we get
\begin{equation*}
\begin{split}
\mathbb{E}^{x_0}(S_k^{++1}+S_k^{--1})
&=  \frac{h_n}{2}\mathbb{E}^{x_0}\left[\langle \partial a(\overline{X}_{t_k}),\nabla_x u(\theta_{k+1},\overline{X}_{t_k})\rangle \right ]\\
&\quad -\mathbb{E}^{x_0}\left[\langle \triangle^\sharp_{k+1}\overline{X} ,
\nabla_x u(\theta_{k+1},\overline{X}_{t_k})\rangle ~\indi{\Omega_k^{+-}\cup\Omega_k^{-+}}\right].\\
&\quad 
\end{split}
\end{equation*}
Proceeding similarly and conditioning $(\triangle^\sharp_{k+1}\overline{X})^2$ 
w.r.t. the past of $(W_t)$ up to time $t_k$, we obtain
\begin{equation*}
\begin{split}
\mathbb{E}^{x_0}(S_k^{++2}+S_k^{--2})
&= \dfrac{1}{2}\mathbb{E}^{x_0}\left[{\rm Tr}[\sigma {\bf H}[u]\sigma^\ast](\theta_{k+1},\overline{X}_{t_k})\right]h_n \\
&\quad
-\dfrac{1}{2}\mathbb{E}^{x_0}\left[(\triangle^\sharp_{k+1}\overline{X})^\ast
{\bf H}[u](\theta_{k+1},\overline{X}_{t_k})\triangle^\sharp_{k+1}\overline{X}~\indi{\Omega_k^{+-}\cup\Omega_k^{-+}}\right]
\end{split}
\end{equation*}
and since $\mathbb{E}^{x_0}(\triangle_{k+1}W)^\alpha=0$ whenever $|\alpha|=3$,
\begin{equation*}
\begin{split}
\mathbb{E}^{x_0}(S_k^{++3}+S_k^{--3})
&= \sum_{|{\alpha}|= 3}\frac{1}{\alpha !}\mathbb{E}^{x_0}\croc{(\triangle^\sharp_{k+1}\overline{X})^\alpha\frac{\partial^{3}u}{\partial x^{\alpha}}(\theta_{k+1},\overline{X}_{t_k})}\\
&\quad
-\sum_{|{\alpha}|= 3}\frac{1}{\alpha !}\mathbb{E}^{x_0}\croc{(\triangle^\sharp_{k+1}\overline{X})^\alpha\frac{\partial^{3}u}{\partial x^{\alpha}}(\theta_{k+1},\overline{X}_{t_k})~\indi{\Omega_k^{+-}\cup\Omega_k^{-+}}}.
\end{split}
\end{equation*}

We have, combining the results of Corollary \ref{cor:control-der} and Proposition \ref{prop:moment-schema},
\begin{equation}
\left | \sum_{|{\alpha}|= 3}\frac{1}{\alpha !}\mathbb{E}^{x_0}\croc{(\triangle^\sharp_{k+1}\overline{X})^\alpha\frac{\partial^{3}u}{\partial x^{\alpha}}(\theta_{k+1},\overline{X}_{t_k})}\right | \leq C\;h_n^{3/2}.
\end{equation}

In addition, and for the same reasons, we have
$$ \mathbb{E}^{x_0}|S_k^{++4}+S_k^{--4}| \leq C\,h_n^2. $$
To summarize the calculations of this subsection, we have obtained
\begin{equation} \label{expansion-Sk}
\begin{split}
&\mathbb{E}^{x_0}S_k =\\
&\mathbb{E}^{x_0}\mathcal{L}u(\theta_{k+1},\overline{X}_{t_k})h_n +\mathbb{E}^{x_0}
\left[\pare{S_k-\langle \triangle^\sharp_{k+1}\overline{X} , \nabla_x u(\theta_{k+1},\overline{X}_{t_k})\rangle }
~\indi{\Omega_k^{+-}\cup\Omega_k^{-+}} \right ] \\
&\!\!\!- \mathbb{E}^{x_0}\Big[\Big(\dfrac{1}{2}(\triangle^\sharp_{k+1}\overline{X})^\ast
{\bf H}[u](\theta_{k+1},\overline{X}_{t_k})\triangle^\sharp_{k+1}\overline{X}\\
&+\sum_{|{\alpha}|= 3}\frac{1}{\alpha !}(\triangle^\sharp_{k+1}\overline{X})^\alpha\frac{\partial^{3}u}{\partial x^{\alpha}}(\theta_{k+1},\overline{X}_{t_k})\Big)~\indi{\Omega_k^{+-}\cup\Omega_k^{-+}}\Big]+O(h_n^{3/2}) \\
&=: \mathbb{E}^{x_0}\mathcal{L}u(\theta_{k+1},\overline{X}_{t_k})h_n
+ {\mathbb E}^{x_0}\mathcal{R}^{(1)}_k - \mathbb{E}^{x_0}\mathcal{R}^{(2)}_k+O(h_n^{3/2}).
\end{split}
\end{equation}
We now estimate the remaining terms $\mathbb{E}^{x_0}\mathcal{R}^{(1)}_k$ and $\mathbb{E}^{x_0}\mathcal{R}^{(2)}_k$.

\subsection{Control of the term $\mathbb{E}^{x_0}\mathcal{R}^{(1)}_k$. Expansion around a well chosen point in $\Gamma$} 
\label{subsec:developpement0}
On the event $\Omega_k^{+-}$ we have that $\overline{X}_{t_{k+1}}$ and $\overline{X}_{t_k}$ are close to $\Gamma$.
On this event, we also have that $\hat{X}_{t_{k+1}}\in D_-$ and $\overline{X}_{t_k}\in D_+$. 
Remember our definition of $(F^{\gamma_+}(x), \pi_{\Gamma}^{\gamma_+}(x))$ for $x\in D_-$.

\subsubsection{Decomposition of $\mathbb{E}^{x_0}\mathcal{R}^{(1)}_k$}

As the function $u$ is continuous across the surface $\Gamma$ at point $\pi_{\Gamma}^{\gamma_+}(x)$, we get
\begin{equation*}
\begin{split}
&\mathbb{E}^{x_0}\left(\pare{S_k-\langle \triangle^\sharp_{k+1}\overline{X} , \nabla_x u(\theta_{k+1},\overline{X}_{t_k})\rangle }~\indi{\Omega_k^{+-}}\right]\\
&= \mathbb{E}^{x_0}\big[\big(\big(u(\theta_{k+1},\overline{X}_{t_{k+1}}) - u(\theta_{k+1}, \pi_{\Gamma}^{\gamma_+}(\hat{X}_{t_{k+1}}))\big)\\
&\;\;+\pare{u(\theta_{k+1}, \pi_{\Gamma}^{\gamma_+}(\hat{X}_{t_{k+1}})) - u(\theta_{k+1},\overline{X}_{t_k})}\big)~\indi{\Omega_k^{+-}}\big]\\
&\quad\quad -\mathbb{E}^{x_0}\left[\langle \triangle^\sharp_{k+1}\overline{X} , \nabla_x u_+(\theta_{k+1},\pi_{\Gamma}^{\gamma_+}(\hat{X}_{t_{k+1}}))\rangle~\indi{\Omega_k^{+-}}\right] \\
&\quad\quad- \mathbb{E}^{x_0}\left[\langle  \triangle^\sharp_{k+1}\overline{X} , \nabla_x u(\theta_{k+1},\overline{X}_{t_k}) -\nabla_x u_+(\theta_{k+1},\pi_{\Gamma}^{\gamma_+}(\hat{X}_{t_{k+1}}))\rangle~\indi{\Omega_k^{+-}}\right]
\end{split}
\end{equation*}
so that
{\small
\begin{equation*}
\begin{split}
&\mathbb{E}^{x_0}\left(\pare{S_k-\langle \triangle^\sharp_{k+1}\overline{X} , \nabla_x u(\theta_{k+1},\overline{X}_{t_k})\rangle }~\indi{\Omega_k^{+-}}\right]\\
=& \left \{\mathbb{E}^{x_0}\left[\langle \overline{X}_{t_{k+1}} - \pi_{\Gamma}^{\gamma_+}(\hat{X}_{t_{k+1}}),\nabla_x u_-(\theta_{k+1},\pi_{\Gamma}^{\gamma_+}(\hat{X}_{t_{k+1}}))\rangle\indi{\Omega_k^{+-}}\right]\right .\\
&\quad\quad -\mathbb{E}^{x_0}\left[\langle \overline{X}_{t_{k}} - \pi_{\Gamma}^{\gamma_+}(\hat{X}_{t_{k+1}}),\nabla_x u_+(\theta_{k+1},\pi_{\Gamma}^{\gamma_+}(\hat{X}_{t_{k+1}})) \rangle~\indi{\Omega_k^{+-}}\right]\\
&\quad\quad \left . -\mathbb{E}^{x_0}\left[\langle \triangle^\sharp_{k+1}\overline{X} , \nabla_x u_+(\theta_{k+1},\pi_{\Gamma}^{\gamma_+}(\hat{X}_{t_{k+1}}))\rangle~\indi{\Omega_k^{+-}}\right]\right \}_{:=L_{k}^{+-1}}\\
&\quad\quad + \left \{ \int_0^1 d\xi \sum_{|{\alpha}|= 2}\frac{(1-\xi)^2}{\alpha !}\mathbb{E}^{x_0}\left[(\hat{X}_{t_{k+1}} - \pi_{\Gamma}^{\gamma_+}(\hat{X}_{t_{k+1}}))^\alpha\right .\right .\\
&\quad\quad\quad\quad \left .\times \frac{\partial^{\alpha}u}{\partial x^{\alpha}}(\theta_{k+1},\pi_{\Gamma}^{\gamma_+}(\hat{X}_{t_{k+1}}) + \xi(\hat{X}_{t_{k+1}} - \pi_{\Gamma}^{\gamma_+}(\hat{X}_{t_{k+1}})))~\indi{\Omega_k^{+-}}\right ]\\
&\quad\quad\quad - \int_0^1 d\xi \sum_{|{\alpha}|= 2}\frac{(1-\xi)^2}{\alpha !}\mathbb{E}^{x_0}\left[(\overline{X}_{t_{k}} - \pi_{\Gamma}^{\gamma_+}(\hat{X}_{t_{k+1}}))^\alpha\right .\\
&\quad\quad\quad\quad \left .\left .\times\frac{\partial^{\alpha}u}{\partial x^{\alpha}}(\theta_{k+1},\pi_{\Gamma}^{\gamma_+}(\hat{X}_{t_{k+1}}) + \xi(\overline{X}_{t_{k}} - \pi_{\Gamma}^{\gamma_+}(\hat{X}_{t_{k+1}})))~\indi{\Omega_k^{+-}}\right ]\right \}_{:= L_{k}^{+-2}}\\
&- \left \{ \mathbb{E}^{x_0}\left[\langle \triangle^\sharp_{k+1}\overline{X} , \nabla_x u_+(\theta_{k+1},\overline{X}_{t_k}) -\nabla_x u_+(\theta_{k+1},\pi_{\Gamma}^{\gamma_+}(\hat{X}_{t_{k+1}}))\rangle~\indi{\Omega_k^{+-}}\right]\right \}_{:=L_{k}^{+-3}}\\
=&L_{k}^{+-1} + L_{k}^{+-2} + L_{k}^{+-3}.
\end{split}
\end{equation*}
}
\subsubsection{Canceling the term $L_{k}^{+-1}$ using the transmission condition}
\label{ssec:cancel-tc}

Observe that due to the fact that
\begin{equation*}
\pare{\hat{X}_{t_{k}} - \pi_{\Gamma}^{\gamma_+}(\hat{X}_{t_{k+1}})} + \pare{\hat{X}_{t_{k+1}} - \hat{X}_{t_{k}}} = \hat{X}_{t_{k+1}} - \pi_{\Gamma}^{\gamma_+}(\hat{X}_{t_{k+1}}).
\end{equation*}
we have that
\begin{equation*}
\begin{split}
L_{k}^{+-1}& = \mathbb{E}^{x_0}\left[\left (\langle \overline{X}_{t_{k+1}} - \pi_{\Gamma}^{\gamma_+}(\hat{X}_{t_{k+1}}),\nabla_x u_-(\theta_{k+1},\pi_{\Gamma}^{\gamma_+}(\hat{X}_{t_{k+1}}))\rangle\right .\right .\\
&\quad\quad\quad\quad\left . \left .-\langle \hat{X}_{t_{k+1}} - \pi_{\Gamma}^{\gamma_+}(\hat{X}_{t_{k+1}}),\nabla_x u_+(\theta_{k+1},\pi_{\Gamma}^{\gamma_+}(\hat{X}_{t_{k+1}})) \rangle\right )~\indi{\Omega_k^{+-}}\right]\\
&=\mathbb{E}^{x_0}\left[F^{\gamma_+}(\hat{X}_{t_{k+1}})\left (\langle -\gamma_-(\pi_{\Gamma}^{\gamma_+}(\hat{X}_{t_{k+1}})),\nabla_x u_-(\theta_{k+1},\pi_{\Gamma}^{\gamma_+}(\hat{X}_{t_{k+1}}))\rangle \right .\right .\\
&\quad\quad\quad\quad\quad \quad\left . \left . -\langle \gamma_+(\pi_{\Gamma}^{\gamma_+}(\hat{X}_{t_{k+1}})),\nabla_x u_+(\theta_{k+1},\pi_{\Gamma}^{\gamma_+}(\hat{X}_{t_{k+1}})) \rangle\right )~\indi{\Omega_k^{+-}}\right]\\
&=0,
\end{split}
\end{equation*}
where we have used the vector problem solved by $(F^{\gamma_+}, \pi_{\Gamma}^{\gamma_+})$ and Equation \eqref{eq:trans-gamma} (i.e. the transmission condition $(\star)$ and the definition of $\gamma_\pm(x)$).
\medskip

\subsubsection{The term $L_k^{+-2}$}
We now turn to the term $L_{k}^{+-2}$.

The term $L_k^{+-2}$ is the sum of two terms. These two terms are treated similarly, so we concentrate only on the first.
Let $\alpha$ such that $|\alpha|= 2$. We have that 
\begin{equation*}
\begin{split}
\mathbb{E}^{x_0}\left[\left |(\overline{X}_{t_{k}} - \pi_{\Gamma}^{\gamma_+}(\hat{X}_{t_{k+1}}))^\alpha\right |~\indi{\Omega_k^{+-}}\right ]&\leq c_1 \mathbb{E}^{x_0}\left[|\overline{X}_{t_{k}} - \pi_{\Gamma}^{\gamma_+}(\hat{X}_{t_{k+1}})|^2~\indi{\Omega_k^{+-}}\right ]\\
&\leq c_2 \mathbb{E}^{x_0}\left[|\triangle^\sharp_{k+1} X|^2~\indi{\Omega_k^{+-}}\right ].
\end{split}
\end{equation*}
The same kind of treatment can be performed for the second term of $L_k^{+-2}$. Conditionning w.r.t ${\cal F}_{t_k}$ and applying the Cauchy-Schwarz inequality in the conditionnal expectation, we find using the result of Lemma \ref{lemme:Ito-control},
\begin{align*}
\begin{split}
|L_k^{+-2}| &\leq C\mathbb{E}^{x_0}\croc{{\mathbb E}^{{\cal F}_{t_k}}\left[\big|\triangle^\sharp_{k+1} X\big|^4\right ]^{1/2}~{\mathbb P}^{{\cal F}_{t_k}}\pare{\Omega_k^{+-}}^{1/2}}\\
&\leq C\,h_n ~{\mathbb E}^{x_0}{\mathbb P}^{{\cal F}_{t_k}}\pare{\Omega_k^{+-}}^{1/2}.
\end{split}
\end{align*}

\subsubsection{The term $L_k^{+-3}$}
For the term $L_k^{+-3}$, we may perform a Taylor's expansion to the term 
\begin{equation*}
\nabla_x u_+(\theta_{k+1},\overline{X}_{t_k}) -\nabla_x u_+(\theta_{k+1},\pi_{\Gamma}^{\gamma_+}(\hat{X}_{t_{k+1}})).
\end{equation*}
Using Corollary \ref{cor:control-der} and the Cauchy-Schwarz inequality, we find that
\begin{align}
|L_k^{+-3}| &\leq C\mathbb{E}^{x_0}\left[ \big|\triangle^\sharp_{k+1}\overline{X}\big| \big|\overline{X}_{t_k} - \pi_{\Gamma}^{\gamma_+}(\hat{X}_{t_{k+1}})\rangle \big| ~\indi{\Omega_k^{+-}}\right]\nonumber\\
&\leq C \mathbb{E}^{x_0}\left[\big|\triangle^\sharp_{k+1} X\big|^2~\indi{\Omega_k^{+-}}\right ].
\end{align}

Finally, as for the term $L_k^{+-2}$, we find that
\begin{equation*}
\begin{split}
|L_k^{+-3}| \leq C\,h_n {\mathbb E}^{x_0}{\mathbb P}^{{\cal F}_{t_k}}\pare{\Omega_k^{+-}}^{1/2}.
\end{split}
\end{equation*}

Using the same method for the other side $\Omega_k^{-+}$, we find that
\begin{equation*}
\begin{split}
\E^{x_0}{\cal R}_k^{(1)} \leq C\,h_n {\mathbb E}^{x_0}\pare{{\mathbb P}^{{\cal F}_{t_k}}\pare{\Omega_k^{+-}}^{1/2} + {\mathbb P}^{{\cal F}_{t_k}}\pare{\Omega_k^{-+}}^{1/2}}.
\end{split}
\end{equation*}

\subsection{Summing up}
The term $\mathbb{E}^{x_0}\mathcal{R}^{(2)}_k$ can be estimated using the same techniques used in the previous section and we omit the details.

Using now the fact that $\partial_t u - {\cal L}u =0$, we finally find that
\begin{equation}
\epsilon^{x_0}_T\leq C \,h_n\,{\mathbb E}^{x_0}\sum_{k=0}^{n-1} \pare{{\mathbb P}^{{\cal F}_{t_k}}\pare{\Omega_k^{+-}}^{1/2} + {\mathbb P}^{{\cal F}_{t_k}}\pare{\Omega_k^{-+}}^{1/2}} + C\sqrt{h_n}.
\end{equation} 
Observe -- using the result of Lemma \ref{lemme:Ito-control} -- that 
\begin{align*}
{\mathbb P}^{{\cal F}_{t_k}}\pare{\Omega_k^{+-}}^{1/2} &= {\mathbb P}^{{\cal F}_{t_k}}\pare{\overline{X}_{t_k}\in D_+, \hat{X}_{t_{k+1}}\in D_-}^{1/2}\\
&\leq {\mathbb P}^{{\cal F}_{t_k}}\pare{||\hat{X}_{t_{k+1}} - \overline{X}_{t_k}||\geq d\pare{\overline{X}_{t_k}, \Gamma}}^{1/2}\\
&\leq K(T)\exp\pare{-\frac{1}{2}\frac{d^2\pare{\overline{X}_{t_k}, \Gamma}}{h_n}}
\end{align*}
and the same kind of inequality holds true for ${\mathbb P}^{{\cal F}_{t_k}}\pare{\Omega_k^{-+}}^{1/2}$.

Finally,
\begin{align*}
\epsilon^{x_0}_T\leq K(T)\,h_n\,{\mathbb E}^{x_0}\sum_{k=0}^{n-1}\exp\pare{-\frac{1}{2}\frac{d^2\pare{\overline{X}_{t_k}, \Gamma}}{h_n}} + C\sqrt{h_n},
\end{align*}
and we conclude the proof of Theorem \ref{thm-conv-schema} using the result of Lemma \ref{lemme:control-couronne}
(note that if we sum up all the dependancies of our constants, we indeed have that $K$ in \eqref{eq-err-schema} depends on
$d$, $\lambda$, $\Lambda$, $u_0$ and~$T$).

\section{Numerical experiments}
\label{sec:num}

 In these examples $d=2$ and the domain $D$ is the open unit disc, i.e.,
$$
D=\{(x_1,x_2)\in\R:\,x_1^2+x_2^2<1\}.$$
Note that the boundary of $D$ is the unit circle $\partial D=\{(x_1,x_2)\in\R:\,x_1^2+x_2^2=1\}$.

The subdomains $D_+$ and $D_-$ are defined by
$$
D_+=\{(x_1,x_2)\in D\;\text{with}\;x_2>0\}\quad\text{and}\quad D_-=\{(x_1,x_2)\in D\;\text{with}\;x_2<0\},$$
so that the interface is $\Gamma=\{(x_1,0)\in \R^2:\,-1\leq x_1\leq 1\}$.

The diffusion matrix is defined by
$$
a(x)=a_+(x)\indi{x\in D_+} + a_-(x)\indi{x\in \bar{D}_-},$$
with 
$$
a_\pm(x)=P^\ast_\pm E_\pm(x)P_\pm$$
where $P_\pm$ are rotation (therefore orthogonal) matrices
$$
P_\pm=\left(\begin{array}{cc}
\cos(\theta_\pm)&-\sin(\theta_\pm)\\
\sin(\theta_\pm)& \cos(\theta_\pm)\\
\end{array}
\right)
$$
(for $\theta_\pm\in[0,2\pi)$), and $E_\pm(x)$ are diagonal matrix-valued functions
$$
E_\pm(x)=\left(\begin{array}{cc}
\lambda^1_\pm+\epsilon_\pm x_2&0\\
0&  \lambda^2_\pm+\epsilon_\pm x_2 \\
\end{array}
\right)$$
where $\lambda^1_\pm,\lambda^2_\pm>0$ and $\epsilon_\pm<\lambda^i_\pm$ for $i=1,2$. Note that this ensures that $a(x)$ satisfies the uniform ellipticity assumption $\mathbf{(E)}$.

We take $\theta_+=\frac{\pi}{4}$, $\theta_-=\frac{\pi}{3}$, $\lambda^1_+=1$, $\lambda^2_+=9$, $\lambda^1_-=2$, $\lambda^2_-=3$ ,
$\epsilon_+=0.5$ and $\epsilon_-=1.9$.
This gives
$$
a_+(x)=\frac 1 2
\left(\begin{array}{cc}
5+0.5x_2 & 4\\
4& 5+0.5x_2\\
\end{array}
\right),
\;\;
a_-(x)=\frac 1 2
\left(\begin{array}{cc}
\frac{11}{4}+1.9x_2 & \frac{\sqrt{3}}{4}\\
\frac{\sqrt{3}}{4} & \frac 9 4  + 1.9x_2\\
\end{array}
\right).
$$

\vspace{0.3cm}

\noindent
{\bf Performing our Transformed Euler Scheme.}

We have the Cholesky decompositions $2a_\pm(x)~=~\sigma_\pm\sigma_\pm^\ast(x)$, with
$$
\sigma_+(x)=
\left(\begin{array}{cc}
\sqrt{5+0.5x_2} & 0\\
4/\sqrt{5+0.5x_2} &\sqrt{5+0.5x_2 -16/(5+0.5x_2)} \\
\end{array}
\right)
$$
and
$$
\sigma_-(x)=
\left(\begin{array}{cc}
\sqrt{\frac{11}{4}+1.9x_2} & 0\\
\frac{\sqrt{3}}{4}/ \sqrt{\frac{11}{4}+1.9x_2} & \sqrt{\frac 9 4  + 1.9x_2-3/(44+30.4x_2)}\\
\end{array}
\right),
$$
so that $2a(x)=\sigma\sigma^\ast(x)$ with $\sigma(x)=\sigma_+(x)\indi{x\in D_+} + \sigma_-(x)\indi{x\in \bar{D}_-}$.
Besides we have 
$$
\partial a(x)=\left(
\begin{array}{c}
0\\
0.25\\
\end{array}
\right)\indi{x\in D_+}
+\left(
\begin{array}{c}
0\\
0.95\\
\end{array}
\right)\indi{x\in \bar{D}_-}.
$$
Note that when the scheme crosses the interface $\Gamma$, we compute the quantities 
$\pi_\Gamma^{\gamma_\pm}(\hat{X}_{t_{k+1}})$ and $F^{\gamma_\pm}(\hat{X}_{t_{k+1}})$ in the following way (we will detail the procedure for
$\pi_\Gamma^{\gamma_+}(\hat{X}_{t_{k+1}})$ and $F^{\gamma_+}(\hat{X}_{t_{k+1}})$). Recall that we have
$$
\hat{X}_{t_{k+1}}-\pi_\Gamma^{\gamma_+}(\hat{X}_{t_{k+1}})=F^{\gamma_+}(\hat{X}_{t_{k+1}})\gamma_+(\pi_\Gamma^{\gamma_+}(\hat{X}_{t_{k+1}})).$$
But here $\nu=(0, 1)^*$ so that for any $x\in\Gamma$
$$
\gamma_+(x)=\frac 1 2
\left(
\begin{array}{c}
4\\
5+0.5x_2\\
\end{array}
\right)
$$
and $\big(\pi_\Gamma^{\gamma_+}(\hat{X}_{t_{k+1}})\big)_2=0$ so that 
$\big(\hat{X}_{t_{k+1}}-\pi_\Gamma^{\gamma_+}(\hat{X}_{t_{k+1}})\big)_2
=\big(\hat{X}_{t_{k+1}}\big)_2$. 
This yields
$$F^{\gamma_+}(\hat{X}_{t_{k+1}})=\frac{\big(\hat{X}_{t_{k+1}}\big)_2}{2.5},$$ 
and then
$$
\pi_\Gamma^{\gamma_+}(\hat{X}_{t_{k+1}})=
\left(
\begin{array}{c}
\big(\hat{X}_{t_{k+1}}\big)_1-F^{\gamma_+}(\hat{X}_{t_{k+1}})\times 2\\
0\\
\end{array}
\right).
$$
Then we have everything in hand to perform our Tranformed Euler Scheme $\overline{X}$.
\vspace{0.3cm}

\noindent
{\bf Comparing with an Euler scheme applied on regularized coefficients.}
 A natural method with which to compare our tranformed scheme is to  regularize first the coefficients and then to perform a standard (i.e. not transformed) Euler scheme.  More precisely consider the operator
 $$
 C^2(\R^d;\R)\ni f\mapsto \cL^\varepsilon f=\nabla \cdot\big(a^\varepsilon \nabla_x f\big)=
 {\rm Tr}\croc{\mathbf{H}[f]a^\varepsilon}+(\partial a^\varepsilon)^* \,\nabla_x f$$
where $a^\varepsilon$ is some smoothed version of $a$ ($\varepsilon$ is the regularization step, see the following discussion about its choice). Then~$\cL^\varepsilon$ is the generator of the solution of the SDE
\begin{equation}
\label{eq:EDS-regu}
dX^\varepsilon_t=\sigma^{\varepsilon}(X^\varepsilon_t)\,dW_t+[\partial a^\varepsilon](X^\varepsilon_t)\,dt,
\end{equation}
where $\sigma^\varepsilon(\sigma^\varepsilon)^*=2a^\varepsilon$. The process $X^\varepsilon$ may be approached by a standard (i.e. not transformed) Euler scheme $\overline{X}^\varepsilon$, with time step $h_n$. 

Let $h_n$ be fixed. In fact $\varepsilon$ will be chosen in function of $h_n$. We are first inspired by the random walk approach proposed in~\cite{stroock1997}. In this later paper  Equation $(3.11)$ indicated that $\varepsilon$ has to be proportional to the square root of the space discretisation step.
Then, using a scaling argument we 
choose~$\varepsilon~=~h_n^{1/4}$. 

Then we set
$$
a^\varepsilon(x)=a(x)\1_{|x_2|>\varepsilon}+A^\varepsilon(x)\1_{|x_2|\leq \varepsilon}$$
where
$$
A^\varepsilon(x)=\frac 1 2
\left(\begin{array}{cc}
\frac{31}{8}-0.7\varepsilon+x_2(\frac{9}{8\varepsilon}+1.2) & \frac{\sqrt{3}}{8}+2+x_2(\frac 2 \varepsilon -\frac{\sqrt{3}}{8\varepsilon})\\
\frac{\sqrt{3}}{8}+2+x_2(\frac 2 \varepsilon -\frac{\sqrt{3}}{8\varepsilon})& \frac {29}{ 8} - 0.7\varepsilon+x_2(\frac{11}{8\varepsilon}+1.2)\\
\end{array}
\right).
$$
Note that the thus defined coefficient $a^\varepsilon$ is continuous and piecewise differentiable.
Then we have
$\partial a^\varepsilon= \partial a(x)\1_{|x_2|>\varepsilon}+\partial A^\varepsilon(x)\1_{|x_2|\leq \varepsilon}$
where
$$
\partial A^\varepsilon (x)=\left(
\begin{array}{c}
\frac 1 \varepsilon -\frac{\sqrt{3}}{16\varepsilon}\\
\frac{11}{16\varepsilon}+0.6\\
\end{array}
\right),
$$
and
$2a^\varepsilon(x)=\sigma^\varepsilon[\sigma^\varepsilon]^*(x)$ with 
$\sigma^\varepsilon(x)=\sigma(x)\1_{|x_2|>\varepsilon}+\Sigma^\varepsilon(x)\1_{|x_2|\leq \varepsilon}$ and $\Sigma^\varepsilon(x)$ being equal to
\small
$$
\left(\begin{array}{cc}
\sqrt{\frac{31}{8}-0.7\varepsilon+x_2(\frac{9}{8\varepsilon}+1.2)} & 0\\
\frac{\frac{\sqrt{3}}{8}+2+x_2(\frac 2 \varepsilon -\frac{\sqrt{3}}{8\varepsilon})}{\sqrt{\frac{31}{8}-0.7\varepsilon+x_2(\frac{9}{8\varepsilon}+1.2)}}
&\sqrt{ \frac {29}{ 8} - 0.7\varepsilon+x_2(\frac{11}{8\varepsilon}+1.2)
-\frac{\big( \frac{\sqrt{3}}{8}+2+x_2(\frac 2 \varepsilon -\frac{\sqrt{3}}{8\varepsilon}) \big)^2}{
\frac{31}{8}-0.7\varepsilon+x_2(\frac{9}{8\varepsilon}+1.2)}
}\\
\end{array}
\right).
$$
\normalsize
With these coefficients it is easy to perform a standard Euler Scheme on the SDE \eqref{eq:EDS-regu}.

\vspace{0.4cm}

We will compare both methods on the two following examples. Benchmarks will be provided by a deterministic approximation of the solutions of the PDE of interest.

\vspace{0.4cm}

\noindent
{\bf Example 1.}
We wish here to treat  the elliptic transmission problem
$$
(\mathcal{E}^0_{\mathrm{T},\text{bounded} \;D})
\left\{
\begin{array}{rcll}
\cL u(x)&=&0&\forall x\in D\\
\\
\langle a_+\nabla_x u_{+}(y)- a_-\nabla u_{-}(y), \nu(y)\rangle &=& 0&\forall y\in \Gamma \\
\\
u(y+)&=&u(y-)&\forall y\in\Gamma\\
\\
u(x)&=&f(x)&\forall x\in\partial D.\\
\end{array}
\right.
$$
We take the function $f$ to be
$$
f(x)=\sin(3x_1)+\cos(4x_2).$$
Consider then on one side our study of the convergence in the parabolic case, and on the other side the Feynman-Kac representation for elliptic PDEs available in the smooth case (see for instance Theorem~5.7.2 in \cite{Karatzas-Shreve-1991}).
One can hope that 
$$
\mathbb{E}^x[f(\overline{X}_{\overline{\tau}})]\xrightarrow[h_n\to 0]{} u(x),$$
where  $\overline{X}$ denotes our scheme and $\overline{\tau}=\inf\{t\geq 0:\,\overline{X}_t\notin D\}$.

We thus compute a Monte Carlo approximation of $\mathbb{E}^x[f(\overline{X}_{\overline{\tau}})]$ on one side (with $N=10^6$ paths). Note that in this Monte Carlo procedure we have used a boundary shifting method, on order to reduce the bias introduced by the approximation of the exit time $\tau=\inf\{ t\geq 0:\,X_t\notin D\}$ by $\overline{\tau}$ (see \cite{gobet-livre-num} Subsection 5.4.3, and the references therein). 

On the other side
$\mathbb{E}^x[f(\overline{X}^\varepsilon_{\overline{\tau}^\varepsilon})]$, with  
$\overline{\tau}^\varepsilon=\inf\{t\geq 0:\,\overline{X}^\varepsilon_t\notin D\}$,
provides another approximation of~$u(x)$ (note that we use again a boundary shifting method).

 Benchmarks are provided by the software FREEFEM with which we compute an approximation of~$u(x)$ by a finite element method, using around $1.5\times 10^6$ triangles and $7\times10^5$ vertices (finite elements basis consists of polynomial functions of order $1$).

\vspace{0.2cm}

Table \ref{tab1} shows the results. It seems that our Transformed Euler scheme converges quicker to the benchmark than the standard Euler scheme applied on regularized coefficients.

\small
\begin{center}
\begin{table}
 \begin{center}
\begin{tabular}{cccc}
\hline
 Point $x$& Finite Element &  Euler Scheme on  &  Transformed  \\
 & by FREEFEM &regularized coefficients  &Euler Scheme\\
&($7.10^5$ vertices)& ($h_n=10^{-n}$, $n=4,5,6$)&($h_n=10^{-n}$, $n=2,4,5,6$)  \\
\\
\hline
$x=(0,0.5)^\ast$& -0.1207& -& -0.136356\\
&&-0.115913& -0.121001\\
&&-0.117946&-0.121299\\
&&-0.118792&-0.120821\\
\hline
$x=(0.9,0.05)^\ast$& 0.92527 & -& 0.824901\\
&& 0.915937&0.924759\\
&& 0.922813 & 0.925370\\
&& 0.923853&0.925389\\
\hline
$x=(-0.3,-0.5)^\ast$& -0.745461 & -& -0.737754\\
&& -0.738184&-0.746226\\
&& -0.739099 & -0.745676\\
&& -0.742611&-0.745829\\
\end{tabular}
\caption{Approximated values of the solution $u(x)$ of $(\mathcal{E}^0_{\mathrm{T},,\text{bounded} \;D})$ at points $x=(0,0.5)^\ast,(0.9,0.05)^\ast,(-0.3,-0.5)^*$ computed with a finite element method ($7.10^5$ vertices), a standard  Euler scheme applied on a regularisation $a^\varepsilon$ of $a$, and our tranformed Euler scheme (with $N=10^6$ Monte Carlo sample, and different values of $h_n$). }
\label{tab1} 
\end{center}
\end{table}
\end{center}

\normalsize

\noindent
{\bf Example 2.}
We now turn to some parabolic example (with the same matrix-valued coefficient~$a$). We consider the following problem $(\mathcal{P}_{\mathrm{T},\text{bounded} \;D})$~:~
$$
\left\{
\begin{array}{rcll}
\partial_tu(t,x)-\mathcal{L}u(t,x)&=&0&\forall (t,x)\in (0,T]\times D \\
\\
\langle a_+\nabla_x u_{+}(t,y)- a_-\nabla_x u_{-}(t,y), \nu(y)\rangle &=& 0&\forall (t,y)\in(0,T]\times\Gamma\quad (\star) \\
\\
u(t,y+)&=&u(t,y-)&\forall (t,y)\in[0,T]\times\Gamma\\
\\
u(t,x)&=&0&\forall (t,x)\in (0,T]\times\partial D\\
\\
u(0,x)&=&u_0(x)&\forall x\in \R^d.\\
\end{array}
\right.
$$
Here we will take $T=0.1$ and
$$
u_0(x)=10*(1-|x|^2).$$
Note that the parabolic problem $(\mathcal{P}_{\mathrm{T},\text{bounded} \;D})$ is posed in a bounded domain, unlike in our theoretical study. But we have found that convenient for numerical purposes.

Note also that $u_0$ belongs to $H^1_0(D)$ and is therefore compatible with the uniform Dirichlet boundary condition in
$(\mathcal{P}_{\mathrm{T},\text{bounded} \;D})$. But 
it does not belong to the domain $\cD(A)$, as it does not satisfy the transmission condition $(\star)$.

Nevertheless one can hope that
$$
\mathbb{E}^x[u_0(\overline{X}_t)\,\1_{t\leq \overline{\tau}}\,]\xrightarrow[h_n\to 0]{} u(t,x)$$
(here we use for example 4.4.5 in \cite{gobet-livre-num} and use again the notation $\overline{\tau}$ of Example 1).

Again we compute a Monte Carlo approximation of $
\mathbb{E}^x[u_0(\overline{X}_t)\,\1_{t\leq \overline{\tau}}\,]$ on one side
and of $\mathbb{E}^x[u_0(\overline{X}^\varepsilon_t)\,\1_{t\leq \overline{\tau}^\varepsilon}\,]$
on the other side
(with $N=10^6$ paths and using again the boundary shifting method).

We use  FREEFEM to compute an approximation of $u(t,x)$ by a finite element method (discretization in space) and a Crank-Nicholson scheme (discretization in time), using around $9\times 10^9$ triangles and $4.5\times10^5$ vertices, and $300$ time steps.

Table \ref{tab2} shows the results, for $t=T$. Again it seems that our transformed Euler scheme converges quicker to the benchmark, even if for some reason it is less obvious at point $x=(0,0.5)^\ast$.

\section*{Acknowledgement}

Research partially supported by Labex Bézout (for Miguel Martinez).

\begin{table}
 \begin{center}
\begin{tabular}{cccc}
\hline
 Point $x$& Finite Element /& Euler Scheme on&  Transformed  \\
& Crank-Nicholson &  regularised coefficients& Euler Scheme \\
&($4.5\times 10^5$ vertices,  &($h_n=10^{-n}$, $n=5,6,7$) &($h_n=10^{-n}$, $n=4,5,6,7$) \\
& $300$ time steps)&&\\
\\
\hline
$x=(0,0.5)^\ast$& 2.26288& -& -\\
&& 2.26766& 2.28299\\
&& 2.26332&2.27562 \\
&& 2.26233 &2.26621   \\
\hline
$x=(0.9,0.05)^\ast$& 0.2564 & -& -\\
&&-&-\\
&& 0.269654 & 0.263835\\
&& 0.263029&0.258807\\
\hline
$x=(-0.3,-0.5)^\ast$& 4.24525 & -& -\\
&&-&-\\
&& 4.23472 & 4.23862\\
&&4.23981&4.24483 \\
\hline
$x=(0,0.05)^\ast$& 4.02857 & -& -\\
&&  4.03452&4.0381\\
& &  4.02488&4.03255 \\
&& 4.02483& 4.02936 \\
\end{tabular}
\caption{Approximated values of the solution $u(T=0.1,x)$ of $(\mathcal{P}_{\mathrm{T},\text{bounded} \;D})$ at points $x=(0,0.5)^\ast,(0.9,0.05)^\ast,(-0.3,-0.5)^*,(0,0.05)^*$ computed with a finite element / Crank-Nicholson scheme method ($4.5\times 10^5$ vertices, $300$ times steps), a standard Euler scheme applied on a regularisation $a^\varepsilon$ of $a$, and our tranformed Euler scheme (with $N=10^6$ Monte Carlo sample, and different values of $h_n$). }
\label{tab2} 
\end{center}
\end{table}



  \bibliographystyle{elsarticle-harv} 
  \bibliography{BIB.bib}


\end{document}